\documentclass[a4paper]{scrartcl}

\usepackage[utf8]{inputenc}
\usepackage[T1]{fontenc}

\usepackage{amsmath, amsthm, amssymb}
\usepackage{mathrsfs}
\usepackage{enumerate}
\usepackage{url}
\usepackage[all,cmtip]{xy}
\usepackage{tikz-cd}
\usepackage{relsize}
\usepackage{longtable}
\usepackage{caption}
\usepackage{multirow}
\usepackage{arydshln}

\numberwithin{equation}{section}

\theoremstyle{plain}
\newtheorem*{mainthm}{Main Theorem}
\newtheorem{theorem}{Theorem}[section]
\newtheorem{proposition}[theorem]{Proposition}

\newtheorem{lemma}[theorem]{Lemma}
\newtheorem{corollary}[theorem]{Corollary}
\newtheorem{claim}[theorem]{Claim}

\theoremstyle{definition}
\newtheorem{definition}[theorem]{Definition}
\newtheorem{example}[theorem]{Example}

\theoremstyle{remark}
\newtheorem{remark}[theorem]{Remark}

\newcommand{\eps}{\varepsilon}
\newcommand{\setsuch}[2]{\left\{ #1 \; \middle| \; #2 \right\}}
\newcommand{\restr}[2]{{\left. #1 \right|}_{#2}}
\newcommand{\subl}{{\mathsmaller{<}}}
\newcommand{\subg}{{\mathsmaller{>}}}
\newcommand{\sube}{{\mathsmaller{=}}}
\newcommand{\suble}{{\mathsmaller{\leq}}}
\newcommand{\subge}{{\mathsmaller{\geq}}}
\newcommand{\ext}{\mathsf{\Lambda}}
\newcommand{\transpose}[1]{#1^{\mathrm t}}

\DeclareMathOperator{\Diag}{Diag}
\DeclareMathOperator{\Aff}{Aff}
\DeclareMathOperator{\Id}{Id}
\DeclareMathOperator{\Ad}{Ad}
\DeclareMathOperator{\GL}{GL}

\DeclareMathOperator{\PSL}{PSL}
\DeclareMathOperator{\SO}{SO}
\DeclareMathOperator{\PSO}{PSO}

\DeclareMathOperator{\PSU}{PSU}
\DeclareMathOperator{\Orth}{O}

\newcommand{\ie}{i.e.\ }
\newcommand{\eg}{e.g.\ }

\newcommand{\longlongrightarrow}{\xrightarrow{\hspace*{1cm}}}

\def\hyph{-\penalty0\hskip0pt\relax} 
\begin{document}

\title{Proper affine actions on semisimple Lie algebras}
\author{Ilia Smilga}
              
\maketitle

\begin{abstract}
For any noncompact semisimple real Lie group $G$, we construct a group of affine transformations of its Lie algebra $\mathfrak{g}$ whose linear part is Zariski-dense in $\Ad G$ and which is free, nonabelian and acts properly discontinuously on~$\mathfrak{g}$.
\end{abstract}

\section{Introduction}
\label{sec:intro}

\subsection{Background and motivation}
\label{sec:background}

The present paper is part of a larger effort to understand discrete groups $\Gamma$ of affine transformations (subgroups of the affine group $\GL_n(\mathbb{R}) \ltimes \mathbb{R}^n$) acting properly discontinuously on the affine space $\mathbb{R}^n$. The case where $\Gamma$ consists of isometries (in other words, $\Gamma \subset \Orth_n(\mathbb{R}) \ltimes \mathbb{R}^n$) is well-understood: a classical theorem by Bieberbach says that such a group always has an abelian subgroup of finite index.

We say that a group $G$ acts \emph{properly discontinuously} on a topological space $X$ if for every compact $K \subset X$, the set $\setsuch{g \in G}{g K \cap K \neq \emptyset}$ is finite. We define a \emph{crystallographic} group to be a discrete group $\Gamma \subset \GL_n(\mathbb{R}) \ltimes \mathbb{R}^n$ acting properly discontinuously and such that the quotient space $\mathbb{R}^n / \Gamma$ is compact. In \cite{Aus64}, Auslander conjectured that any crystallographic group is virtually solvable, that is, contains a solvable subgroup of finite index. Later, Milnor \cite{Mil77} asked whether this statement is actually true for any affine group acting properly discontinuously. The answer turned out to be negative: Margulis \cite{Mar83, Mar87} gave a nonabelian free group of affine transformations with linear part Zariski-dense in $\SO(2, 1)$, acting properly discontinuously on $\mathbb{R}^3$. On the other hand, Fried and Goldman \cite{FG83} proved the Auslander conjecture in dimension 3 (the cases $n=1$ and $2$ are easy). Recently, Abels, Margulis and Soifer \cite{AMS13} proved it in dimension $n \leq 6$. See \cite{AbSur} for a survey of already known results.

Margulis's counterexample was also generalized by Abels et al. in \cite{AMS02} to subgroups of $\SO(2n+2,2n+1)$ for all values of $n$. The author improved this result in \cite{Smi13} by giving an explicit construction of associated fundamental domains. (For Margulis's original counterexample, this had been done by Drumm in \cite{Dru92, Dru93}.) However, as far as I know, no other counterexamples to the Milnor conjecture were known until today. In this paper, we construct another family of counterexamples. Here is the result we prove:

\begin{mainthm}
Let $G$ be any noncompact semisimple real Lie group. Consider the "affine group" $G \ltimes \mathfrak{g}$, for the adjoint action of $G$ on its Lie algebra $\mathfrak{g}$. Then there is a subgroup $\Gamma \subset G \ltimes \mathfrak{g}$ whose linear part is Zariski-dense in $G$ and that is free, nonabelian and acts properly discontinuously on the affine space corresponding to $\mathfrak{g}$.
\end{mainthm}

The general strategy of the proof comes from Margulis's original paper \cite{Mar87}; some ideas were also inspired by \cite{AMS11}. (Since the neutral component of $\SO(2, 1)$ acting on $\mathbb{R}^3$ is isomorphic to $\PSL_2(\mathbb{R})$ acting on $\mathfrak{sl}_2(\mathbb{R})$, Margulis's first example is indeed a particular case of this theorem.) Like Margulis, we introduce for some affine maps $g$ an invariant that measures the translation part of $g$ along its neutral space $A^\sube_g$ (defined later). The key part of our argument, just as in \cite{Mar87}, it to show that under some conditions, the invariant of the product of two maps is roughly equal to the sum of their invariants (Proposition \ref{invariant_additivity}). There are two difficulties that were not present in \cite{Mar87}.

First, while the original Margulis invariant was a scalar, our invariant is a vector. To define it properly, we need to introduce some canonical identifications between different spaces $A^\sube_g$, and then follow the transformations of the canonical representative of some vector living in one of these spaces as it gets projected to other spaces.

Second, it turns out that in the general case, $g$ restricted to $A^\sube_g$ is not always a pure translation. It sometimes has a rotation part, but that part is always confined to a proper vector subspace of $A^\sube_g$. The argument still works, but becomes more complicated.

Another novelty of this paper is the notion of a $C$-non-degenerate pair of spaces, which, in the case of affine spaces, encompasses both a quantitative measure of transversality and an upper bound on the distance of these spaces from the origin. It makes the proofs somewhat clearer and simpler.

\subsection{Plan of the paper}
\label{sec:plan}

In Section \ref{sec:basic}, we give some definitions and basic algebraic and metric properties. In Subsection \ref{sec:affine}, we replace the affine space by a linear space $\hat{\mathfrak{g}}$ with one more dimension, more practical to work with; and we define, for every element of the group $G \ltimes \mathfrak{g}$, a family of "dynamical" vector and affine subspaces. In Subsection \ref{sec:lie}, we define some classical subalgebras of $\mathfrak{g}$, including the centralizer $\mathfrak{l}$ of a Cartan subspace. In Subsection \ref{sec:algebraic}, we give some basic algebraic properties: we relate the dynamical subspaces of an $\mathbb{R}$-regular map (see Definition \ref{Rregular}) with the classical subalgebras, and we show that for every such map, the "geometry of the problem" is essentially given by a pair of transverse affine minimal parabolic algebras. In Subsection \ref{sec:quasi-translations}, we introduce an important class of automorphisms of the affine space parallel to $\mathfrak{l}$, called quasi-translations. In Subsection \ref{sec:canonical}, we use the previous two subsections to identify (up to quasi-translation) different pairs of transverse affine minimal parabolic algebras, and to show that these identifications are "natural"; this allows us to define a generalized Margulis invariant (which is a vector). In Subsection \ref{sec:metric}, we introduce a Euclidean metric on the "extended affine space" $\hat{\mathfrak{g}}$, and use it to define two important things: the notion of a $C$-non-degenerate pair of transverse affine minimal parabolic algebras (which means that we may pretend that they are perpendicular and err by no more than some function of $C$), and the contraction strength of an $\mathbb{R}$-regular map. In Subsection \ref{sec:affine_to_linear}, we relate these metric properties of an element of $G \ltimes \mathfrak{g}$ and those of its linear part.

In Section \ref{sec:regular_product}, we show that the product of two $\mathbb{R}$-regular maps "in general position" is still $\mathbb{R}$-regular, and relate the geometry and contraction strength of the product to the relative geometry and contraction strengths of the factors. We do this by examining the dynamics of these maps acting on some exterior power $\ext^p \hat{\mathfrak{g}}$. This section is more or less a generalization of Section 3 of the author's earlier paper \cite{Smi13}, with very similar proofs.

Section \ref{sec:additivity} contains the key part of our argument. We show that under suitable hypotheses, the Margulis invariant of a product of two $\mathbb{R}$-regular maps is approximately equal to the sum of their Margulis invariants. We also relate the Margulis invariants of a map and of its inverse.

In Section \ref{sec:induction}, we use induction to show a similar result for the product of an arbitrary number of maps.

In Section \ref{sec:construction}, we construct a group satisfying the Main Theorem. As generators, we take a family of $\mathbb{R}$-regular, strongly contracting maps in general position with suitable Margulis invariants. Using the result of the previous section, we show that elements of the group have Margulis invariants that grow unboundedly, which turns out (by Lemma \ref{properly_discontinuous}) to ensure a properly discontinuous action.

\section{Preliminary definitions and properties}
\label{sec:basic}

We fix a noncompact semisimple real Lie group $G$. Without loss of generality, we may assume that $G$ is connected with trivial center. We see the group $G$ as a group of automorphisms of $\mathfrak{g}$, via the adjoint representation; in other words, we identify the abstract group $G$ with the linear group $\Ad G \subset \GL(\mathfrak{g})$. Let $\mathfrak{g}_{\Aff}$ be the affine space corresponding to $\mathfrak{g}$. The group of affine transformations of $\mathfrak{g}_{\Aff}$ whose linear part lies in~$G$ may then be written $G \ltimes \mathfrak{g}$ (where $\mathfrak{g}$ stands for the group of translations).

\begin{remark}
As $\mathfrak{g}$ is the tangent space to $G$ at the neutral element, the underlying space of the group $G \ltimes \mathfrak{g}$ is actually the tangent bundle $TG$. In particular, if $\Gamma$ is some abstract group, any representation $\rho_{\Aff}: \Gamma \to G \ltimes \mathfrak{g}$ can be seen as an infinitesimal deformation of the representation $\rho: \Gamma \to G$ corresponding to its linear part. This paper makes no use of this remark; see however the work of Danciger, Gu\'eritaud and Kassel \cite{DGK15, DGKpre} for a lot of interesting results derived from this idea.
\end{remark}

\subsection{Extended affine space and dynamical subspaces}
\label{sec:affine}
We begin with a few definitions.

We choose once and for all a point of $\mathfrak{g}_{\Aff}$ that we take as an origin; we call $\mathbb{R}_0$ the one-dimensional vector space formally generated by this point, and we set $\hat{\mathfrak{g}} := \mathfrak{g} \oplus \mathbb{R}_0$ the \emph{extended affine space} corresponding to~$\mathfrak{g}$. Then $\mathfrak{g}_{\Aff}$ is the affine hyperplane "at height~1" of this space, and $\mathfrak{g}$ is the corresponding vector hyperplane:
\[
\mathfrak{g}       = \mathfrak{g} \times \{0\} \subset \mathfrak{g} \times \mathbb{R}_0; \qquad
\mathfrak{g}_{\Aff} = \mathfrak{g} \times \{1\} \subset \mathfrak{g} \times \mathbb{R}_0.
\]
Any affine map $g$ with linear part $\ell(g)$ and translation vector $v$, defined on $\mathfrak{g}_{\Aff}$ by
\[g: x \mapsto \ell(g)(x)+v,\]
can be extended in a unique way to a linear map defined on $\hat{\mathfrak{g}}$, given by the matrix
\[\begin{pmatrix} \ell(g) & v \\ 0 & 1 \end{pmatrix}.\]
This gives a natural action of the affine group $G \ltimes \mathfrak{g}$ on the vector space $\hat{\mathfrak{g}}$.

We define an \emph{extended affine subspace} of $\hat{\mathfrak{g}}$ to be a vector subspace of $\hat{\mathfrak{g}}$ not contained in $\mathfrak{g}$. There is a one-to-one correspondence between extended affine subspaces of $\hat{\mathfrak{g}}$ and affine subspaces of $\mathfrak{g}_{\Aff}$ of dimension one less. For any extended affine subspace $A$ (or $A_1$, $A_g$ etc.), we denote by $V$ (or $V_1$, $V_g$ etc.) the space $A \cap \mathfrak{g}$ (which is the linear part of the corresponding affine space $A \cap \mathfrak{g}_{\Aff}$).

By abuse of terminology, elements of the normal subgroup $\mathfrak{g} \lhd G \ltimes \mathfrak{g}$ will still be called \emph{translations}, even though we shall see them mostly as endomorphisms of $\hat{\mathfrak{g}}$ (so that they are formally transvections). For any vector $v \in \mathfrak{g}$, we write $\tau_{v}$ the corresponding translation.

For every $g \in G \ltimes \mathfrak{g}$, we decompose $\hat{\mathfrak{g}}$ into a direct sum of three spaces
\[\hat{\mathfrak{g}} = V^{\subg}_{g} \oplus A^{\sube}_{g} \oplus V^{\subl}_{g},\]
called \emph{dynamical subspaces} of $g$, that are all stable by $g$ and such that all eigenvalues $\lambda$ of the restriction of $g$ to $V^{\subg}_{g}$ (resp. $A^{\sube}_{g}$, $V^{\subl}_{g}$) satisfy $|\lambda| > 1$ (resp. $|\lambda| = 1$, $|\lambda| < 1$). We also define $A^{\subge}_{g} := V^{\subg}_{g} \oplus A^{\sube}_{g}$ and $A^{\suble}_{g} := V^{\subl}_{g} \oplus A^{\sube}_{g}$.

In this case, we have of course $V^{\subg}_{g} \subset \mathfrak{g}$ and $V^{\subl}_{g} \subset \mathfrak{g}$ but $A^{\sube}_{g} \not\subset \mathfrak{g}$ (which justifies the choice of the letters $A$ and $V$). It follows that
\[\mathfrak{g} = V^{\subg}_{g} \oplus V^{\sube}_{g} \oplus V^{\subl}_{g},\]
where $V^{\sube}_g$ means $A^{\sube}_g \cap \mathfrak{g}$ according to our convention.

\begin{definition}
\label{Rregular}
An element $g \in G \ltimes \mathfrak{g}$ is said to be \emph{$\mathbb{R}$-regular} if its linear part is $\mathbb{R}$-regular, \ie if the dimension of the space $A^{\sube}_{g}$ (or of its linear part $V^{\sube}_{g}$) is the lowest possible.
\end{definition}
By contrast, when $g$ is a translation, we have $V^{\subg}_{g} = V^{\subl}_{g} = 0$ and $A^{\sube}_{g}$ (resp. $V^{\sube}_{g}$) is the whole space $\hat{\mathfrak{g}}$ (resp. $\mathfrak{g}$).

\subsection{Lie algebra structure}
\label{sec:lie}

Now we introduce a few classical subalgebras of $\mathfrak{g}$ (defined for instance in Knapp's book \cite{Kna96}, though our terminology and notation differ slightly from his). Their value is that if an element $g \in G \ltimes \mathfrak{g}$ is $\mathbb{R}$-regular, then its dynamical subspaces are, up to conjugacy, equal to some of these subalgebras (see Corollary \ref{dynamical_spaces_description}).

We choose in $\mathfrak{g}$:
\begin{itemize}
\item a Cartan involution $\theta$. Then we have the corresponding Cartan decomposition $\mathfrak{g} = \mathfrak{k} \oplus \mathfrak{q}$, where we call $\mathfrak{k}$ the space of fixed points of $\theta$ and $\mathfrak{q}$ the space of fixed points of $-\theta$. We call $K$ the maximal compact subgroup with Lie algebra $\mathfrak{k}$.
\item a \emph{Cartan subspace} $\mathfrak{a}$ compatible with $\theta$ (that is, a maximal abelian subalgebra of $\mathfrak{g}$ among those contained in $\mathfrak{q}$). We set $A := \exp \mathfrak{a}$.
\item a system $\Sigma^+$ of positive restricted roots in $\mathfrak{a}^*$. Recall that a \emph{restricted root} is a nonzero element $\alpha \in \mathfrak{a}^*$ such that the root space
\[\mathfrak{g}_\alpha := \setsuch{Y \in \mathfrak{g}}{\forall X \in \mathfrak{a},\; [X, Y] = \alpha(X)Y}\]
is nontrivial. They form a root system $\Sigma$; a system of positive roots $\Sigma^+$ is a subset of $\Sigma$ contained in a half-space and such that $\Sigma = \Sigma^+ \sqcup -\Sigma^+$. We call
\[\mathfrak{a}^+ := \setsuch{X \in \mathfrak{a}}{\forall \alpha \in \Sigma^+,\; \alpha(X) > 0}\]
the corresponding (open) Weyl chamber of $\mathfrak{a}$.
\end{itemize}
Then we call:
\begin{itemize}
\item $M$ the centralizer of $\mathfrak{a}$ in $K$, $\mathfrak{m}$ its Lie algebra.
\item $L$ the centralizer of $\mathfrak{a}$ in $G$, $\mathfrak{l}$ its Lie algebra. It is clear that $\mathfrak{l} = \mathfrak{a} \oplus \mathfrak{m}$, and well known (see \eg \cite{Kna96}, Proposition 7.82a) that $L = MA$.
\item $\mathfrak{n}^+$ (resp. $\mathfrak{n}^-$) the sum of the restricted root spaces of $\Sigma^+$ (resp. of $-\Sigma^+$).
\item $\mathfrak{p}^+ := \mathfrak{l} \oplus \mathfrak{n}^+$ and $\mathfrak{p}^- := \mathfrak{l} \oplus \mathfrak{n}^-$ the corresponding minimal parabolic algebras.
\item $\hat{\mathfrak{l}}$, $\hat{\mathfrak{p}}^+$ and $\hat{\mathfrak{p}}^-$ the vector extensions of the affine subspaces of $\mathfrak{g}_{\Aff}$ parallel respectively to $\mathfrak{l}$, $\mathfrak{p}^+$ and $\mathfrak{p}^-$ and passing through the origin. In other words:
\[\hat{\mathfrak{l}} := \mathfrak{l} \oplus \mathbb{R}_0 \text{ and } \hat{\mathfrak{p}}^\pm := \mathfrak{p}^\pm \oplus \mathbb{R}_0.\]
\end{itemize}

It is convenient for us to define a \emph{minimal parabolic algebra} (abbreviated as m.p.a. in the sequel) in $\mathfrak{g}$ as the image of $\mathfrak{p}^+$ (or $\mathfrak{p}^-$) by any element of~$G$. Similarly, we define an \emph{affine m.p.a.} in $\hat{\mathfrak{g}}$ as the image of $\hat{\mathfrak{p}}^+$ (or $\hat{\mathfrak{p}}^-$) by any element of $G \ltimes \mathfrak{g}$. Equivalently, a subspace $\hat{\mathfrak{p}}_1 \subset \hat{\mathfrak{g}}$ is an affine m.p.a. iff it is not contained in $\mathfrak{g}$ and its linear part $\hat{\mathfrak{p}}_1 \cap \mathfrak{g}$ is a m.p.a.

We say that two m.p.a.'s (resp. affine m.p.a.'s) are \emph{transverse} if their intersection has the lowest possible dimension (namely $\dim \mathfrak{l}$, resp. $\dim \mathfrak{l} + 1$).

\begin{example} An important special case is $G = \PSL_n(\mathbb{R})$. In this case we may take as~$\theta$ the involution $X \mapsto -\transpose{X}$, as $\mathfrak{a}$ the set of all (traceless) diagonal matrices, and as~$\Sigma^+$ the set of all roots $e_i - e_j$ such that $i < j$. Then:
\begin{itemize}
\item $\mathfrak{k}$ (resp. $\mathfrak{q}$) is the set of traceless antisymmetric (resp. symmetric) matrices, and $K = \PSO_n(\mathbb{R})$;
\item $\mathfrak{a}^+$ is the set of traceless matrices of the form $\Diag(\lambda_1, \ldots, \lambda_n)$ with $\lambda_1 > \cdots > \lambda_n$;
\item $\mathfrak{m}$ is trivial; $\mathfrak{l}$ is equal to $\mathfrak{a}$; $A$ (resp. $M$, $L$) is the group of diagonal matrices with determinant~1 whose coefficients are positive (resp. equal to $\pm 1$, arbitrary);
\item $\mathfrak{n}^+$ (resp. $\mathfrak{n}^-$) is the set of traceless upper (resp. lower) triangular matrices with vanishing diagonal coefficients;
\item $\mathfrak{p}^+$ (resp. $\mathfrak{p}^-$) is the set of all traceless upper (resp. lower) triangular matrices.
\end{itemize}
\end{example}
\begin{example} Another interesting example is $G = \PSO^+(n, 1)$, so that
\[\mathfrak{g} = \mathfrak{so}(n, 1) =
\setsuch{
\left(\begin{array}{cc|c}
\multicolumn{2}{c|}{\multirow{2}{*}{$A$}} & \multirow{2}{*}{$B$}\\
 & & \\
\hline
\multicolumn{2}{c|}{\transpose{B}} & 0
\end{array}\right)
}{B \in \mathbb{R}^n,\; A = -\transpose{A}}.\]
In this case we may take:
\begin{itemize}
\item as $\theta$ the map
\[\theta:
\left(\begin{array}{cc|c}
\multicolumn{2}{c|}{\multirow{2}{*}{$A$}} & \multirow{2}{*}{$B$}\\
 & & \\
\hline
\multicolumn{2}{c|}{\transpose{B}} & 0
\end{array}\right)
  \mapsto
\left(\begin{array}{cc|c}
\multicolumn{2}{c|}{\multirow{2}{*}{$A$}} & \multirow{2}{*}{$-B$}\\
 & & \\
\hline
\multicolumn{2}{c|}{-\transpose{B}} & 0
\end{array}\right);
\]
\item as $\mathfrak{a}$ the line $\mathbb{R}X$ generated by the vector
\[X := \left(\begin{array}{cccc|c}
\multicolumn{4}{c|}{\multirow{4}{*}{$0$}} & 0\\
 & & & & \vdots \\
 & & & & 0 \\
 & & & & 1 \\
\hline
0 & \dots & 0 & 1 & 0
\end{array}\right);\]
\item as $\Sigma^+$ the unique restricted root which is positive on $X$.
\end{itemize}
With these choices:

\begin{itemize}
\addtolength{\itemsep}{.5\baselineskip}
\item
$\mathfrak{k} = \left\{\left(\begin{array}{cc|c}
\multicolumn{2}{c|}{\multirow{2}{*}{$*$}} & \multirow{2}{*}{$0$}\\
 & & \\
\hline
\multicolumn{2}{c|}{0} & 0
\end{array}\right)\right\} \cap \mathfrak{g}$; \;
$\mathfrak{q} = \left\{\left(\begin{array}{cc|c}
\multicolumn{2}{c|}{\multirow{2}{*}{$0$}} & \multirow{2}{*}{$*$}\\
 & & \\
\hline
\multicolumn{2}{c|}{*} & 0
\end{array}\right)\right\} \cap \mathfrak{g}$; \;
$K \simeq \PSO_n(\mathbb{R})$;
\item $\mathfrak{a}^+$ is the ray formed by positive multiples of $X$;
\item
$\mathfrak{m} = \setsuch{\left(\begin{array}{ccc:c|c}
\multicolumn{3}{c:}{\multirow{3}{*}{$A'$}} & 0 & 0\\
  &       &   & \vdots & \vdots \\
  &       &   & 0 & 0 \\
\hdashline
0 & \dots & 0 & 0 & 0 \\
\hline
0 & \dots & 0 & 0 & 0
\end{array}\right)}
{A' = - \transpose{(A')}}
\simeq \mathfrak{so}_{n-1}(\mathbb{R})$;
\item $A = \setsuch{\exp(t X)}{t \in \mathbb{R}}$;\; $M \simeq \PSO_{n-1}(\mathbb{R})$ turns out to be connected in this case;\; $L$ is the direct product of $A$ and $M$.
\item
$\mathfrak{n}^+ = \setsuch{\left(\begin{array}{ccc:c|c}
\multicolumn{3}{c:}{\multirow{2}{*}{$0$}} & \multirow{2}{*}{$-B'$} & \multirow{2}{*}{$B'$}\\
  &    & & \\
\hdashline
\multicolumn{3}{c:}{\transpose{(B')}} & 0 & 0 \\
\hline
\multicolumn{3}{c:}{\transpose{(B')}} & 0 & 0
\end{array}\right)}
{B' \in \mathbb{R}^{n-1}}$;
\item
$\mathfrak{n}^- = \setsuch{\left(\begin{array}{cc:c|c}
\multicolumn{2}{c:}{\multirow{2}{*}{$0$}} & \multirow{2}{*}{$B'$} & \multirow{2}{*}{$B'$}\\
  &    & & \\
\hdashline
\multicolumn{2}{c:}{-\transpose{(B')}} & 0 & 0 \\
\hline
\multicolumn{2}{c:}{\transpose{(B')}} & 0 & 0
\end{array}\right)}
{B' \in \mathbb{R}^{n-1}}$.
\end{itemize}
\end{example}
Note that $G = \PSO^+(2,1) = \SO^+(2,1) \simeq \PSL_2(\mathbb{R})$ is a particular case of both examples.

\subsection{Basic algebraic properties}
\label{sec:algebraic}

We have the following algebraic facts.
\begin{claim}
\label{Asubge_is_ampa}
Let $g \in G \ltimes \mathfrak{g}$.
\begin{enumerate}[(i)]
\item The map $g$ is $\mathbb{R}$-regular iff it is conjugate (by an element of $G \ltimes \mathfrak{g}$) to a product $\tau_v m \exp(a)$ with $v \in \mathfrak{l}$, $m \in M$ and $a \in \mathfrak{a}^+$ (here we identify the subgroup of the affine group $G \ltimes \mathfrak{g}$ fixing the "origin"~$\mathbb{R}_0$ with the linear group $G$).
\item In that case, $A^{\subge}_{g}$ and $A^{\suble}_{g}$ are transverse affine m.p.a.'s.
\item Moreover, in that case $V^{\subg}_{g}$ (resp. $V^{\subl}_{g}$) is uniquely determined by~$A^{\subge}_{g}$ (resp. by $A^{\suble}_{g}$), as the nilradical of its linear part.
\end{enumerate}
\end{claim}
\begin{proof} \mbox{ }
\begin{enumerate}[(i)]
\item Let us show that for any $g \in G \ltimes \mathfrak{g}$, we have $\dim A^{\sube}_g \geq \dim \hat{\mathfrak{l}}$, with equality (\ie $\mathbb{R}$-regularity of $g$) iff $g$ has the required form.

Using the Jordan decomposition (see \eg \cite{Ebe96}, Theorem 2.19.24), we may decompose~$g$ in a unique way as a product $g = \tau_v g_h g_e g_u$, where $v$ is some vector in $\mathfrak{g}$, $g_h \in G$ is hyperbolic (semisimple with positive real eigenvalues), $g_e \in G$ is elliptic (semisimple with eigenvalues of modulus~1), $g_u \in G$ is unipotent (some power of~$g_u - \Id$ is zero), and the last three maps commute with each other. Up to conjugation, we may suppose that $A^\sube_g$ passes through the origin $\mathbb{R}_0$, which means that $v \in V^\sube_g$ (in fact we could even assume that $v$ belongs to the actual $1$-eigenspace).

Since $g_e$ and $g_u$ have all eigenvalues of modulus~1 and commute with $g_h$, we have $A^{\sube}_g = A^{\sube}_{g_h}$. Up to conjugation, we may suppose that $g_h = \exp a$ where $a$ is an element of $\mathfrak{a}$, and even more specifically, of the closure of $\mathfrak{a}^+$. Then clearly the space $A^{\sube}_{g_h}$ is the sum of~$\hat{\mathfrak{l}}$ and of any restricted root spaces $\mathfrak{g}_\alpha$ such that the value $\alpha(a)$ happens to vanish. This shows that $A^{\sube}_{g_h}$ contains $\hat{\mathfrak{l}}$, and that equality occurs iff $a \in \mathfrak{a}^+$.

Clearly if $g$ has the required form, then by uniqueness of the Jordan decomposition we have $g_h = \exp a$, $g_e = m$ and $g_u = 1$, so that $a \in \mathfrak{a}^+$. Conversely, suppose that $a \in \mathfrak{a}^+$; let us show that $g$ has the required form. We start with the observation that any two distinct Weyl chambers are always disjoint; thus the only conjugate of $\mathfrak{a}^+$ containing $a$ is $\mathfrak{a}^+$ itself. It follows that $Z_G(a) = Z_G(\mathfrak{a}) = L$. Then $g_e$ is an elliptic element of $L$, hence an element of $M$; $g_u$ is a unipotent element of $L$, hence equal to~1; and $v \in V^\sube_g = \mathfrak{l}$. It follows that $g = \tau_v g_e \exp a$ with $v \in \mathfrak{l}$, $g_e \in M$ and $a \in \mathfrak{a}^+$, as required.

\item By the previous point, up to conjugation, we may suppose that $g = \tau_v m \exp a$ with $v \in \mathfrak{l}$, $m \in M$ and $a \in \mathfrak{a}^+$. But then clearly $A^\subge_g = A^\subge_{\exp a} = \hat{\mathfrak{p}}^+$ and similarly $A^\suble_g = \hat{\mathfrak{p}}^-$.

\item If $g$ is of the form $\tau_v m \exp a$, then, similarly, we have $V^\subg_g = \mathfrak{n}^+$; and we know that $\mathfrak{n}^+$ is the nilradical (largest nilpotent ideal) of $\mathfrak{p}^+$. Hence for any $\mathbb{R}$-regular $g$, $V^\subg_g$ is the nilradical of $V^\subge_g$, which is the linear part of $A^\subge_g$ (in other words $V^\subge_g = A^\subge_g \cap \mathfrak{g}$). Similarly, $V^\subl_g$ is the nilradical of the linear part of $A^\suble_g$. \qedhere
\end{enumerate}
\end{proof}
\begin{corollary}
\label{dynamical_spaces_description}
For every $\mathbb{R}$-regular map $g \in G \ltimes \mathfrak{g}$, there is a "canonizing" map $\phi \in G \ltimes \mathfrak{g}$ such that:
\begin{align*}
& & \phi(V^\sube_g)  &= \mathfrak{l}   & \phi(A^\sube_g)  &= \hat{\mathfrak{l}}\\
\phi(V^\subg_g)  &= \mathfrak{n}^+ & \phi(V^\subge_g) &= \mathfrak{p}^+ & \phi(A^\subge_g) &= \hat{\mathfrak{p}}^+\\
\phi(V^\subl_g)  &= \mathfrak{n}^- & \phi(V^\suble_g) &= \mathfrak{p}^- & \phi(A^\suble_g) &= \hat{\mathfrak{p}}^-;
\end{align*}
and any map $\phi$ satisfying the last two equalities on the right satisfies all eight of them.
\end{corollary}
\begin{example}~
\begin{itemize}
\item For $G = \PSL_n(\mathbb{R})$, a map $g \in G \ltimes \mathfrak{g}$ is $\mathbb{R}$-regular iff its linear part, seen as an automorphism of $\mathbb{R}^n$ (not of $\mathfrak{g}$ as by our usual convention), has real eigenvalues with distinct absolute values.
\item For $G = \PSO^+(n, 1)$, a map $g \in G \ltimes \mathfrak{g}$ is $\mathbb{R}$-regular iff its linear part, seen as an isometry of the hyperbolic space $\mathbb{H}^n$, is \emph{loxodromic} (acts on the ideal boundary with exactly two fixed points).
\end{itemize}
\end{example}
\begin{claim}
\label{pair_transitivity}
Any pair of transverse m.p.a.'s (resp. transverse affine m.p.a.'s) may be sent to $(\mathfrak{p}^+, \mathfrak{p}^-)$ (resp. $(\hat{\mathfrak{p}}^+, \hat{\mathfrak{p}}^-)$) by some element of $G$ (resp. of $G \ltimes \mathfrak{g}$).
\end{claim}
\begin{proof}
Let us prove the linear version; the affine version follows immediately. Let $(\mathfrak{p}_1, \mathfrak{p}_2)$ be such a pair. By definition, for $i = 1, 2$, we may write $\mathfrak{p_i} = \phi_i(\mathfrak{p}^+)$ for some $\phi_i \in G$. Let us apply the Bruhat decomposition to the map $\phi_1^{-1}\phi_2$: we may write
\[\phi_1^{-1}\phi_2 = p_1wp_2,\]
where $p_1, p_2$ belong to the minimal parabolic subgroup $P^+ := N_G(\mathfrak{p}^+)$, and $w$ is an element of the Weyl group $W := N_G(\mathfrak{a}) / Z_G(\mathfrak{a})$ (see \eg \cite{Kna96}, Theorem 7.40). Let $\phi := \phi_1p_1 = \phi_2p_2^{-1}w^{-1}$; then we have
\[\mathfrak{p}_1 = \phi(\mathfrak{p}^+) \text{ and } \mathfrak{p}_2 = \phi(w\mathfrak{p}^+).\]
It follows that $w\mathfrak{p}^+$ is transverse to $\mathfrak{p}^+$. This occurs iff $w$ is equal to $w_0$, the longest element of the Weyl group; but $w_0\mathfrak{p}^+ = \mathfrak{p}^-$. Thus $\mathfrak{p}_1 = \phi(\mathfrak{p}^+)$ and $\mathfrak{p}_2 = \phi(\mathfrak{p}^-)$ as required.
\end{proof}
\begin{claim}
\label{u_l_invariant_d_fixed}
Any map $\phi \in G \ltimes \mathfrak{g}$ leaving invariant both $\hat{\mathfrak{p}}^+$ and $\hat{\mathfrak{p}}^-$ belongs to the group $L \ltimes \mathfrak{l}$.
\end{claim}
\begin{proof}
It is well-known (this follows for example from \cite{Kna96}, Lemma 7.64) that $N_G(\mathfrak{p}^+) \cap N_G(\mathfrak{p}^-) = Z_G(\mathfrak{a}) = L$; so the linear part of such a map $\phi$ must lie in $L$. Since $\phi$ leaves invariant the space $\hat{\mathfrak{p}}^+ \cap \hat{\mathfrak{p}}^- = \hat{\mathfrak{l}}$, its translation part must lie in $\mathfrak{l}$.
\end{proof}

\subsection{Quasi-translations}
\label{sec:quasi-translations}

In this subsection, we develop upon Claim \ref{u_l_invariant_d_fixed}: we study the action of elements of $L \ltimes \mathfrak{l}$ on the space $\hat{\mathfrak{l}}$.
\begin{definition}
\label{quasi-translation_def}
A \emph{quasi-translation} is any affine automorphism of~$\hat{\mathfrak{l}}$ induced by an element of the group $L \ltimes \mathfrak{l}$.
\end{definition}

Let us explain and justify this terminology. We define $Z := Z(L)$ to be the center of $L$, $D := [L, L]$ to be its derived subgroup, and $\mathfrak{z}$ and $\mathfrak{d}$ to be the corresponding Lie algebras. It is well-known that $L$ is reductive, hence we may write $\mathfrak{l} = \mathfrak{z} \oplus \mathfrak{d}$.

\begin{remark}
\label{a_vs_z}
Since $\mathfrak{l}$ may also be decomposed as $\mathfrak{a} \oplus \mathfrak{m}$ and $\mathfrak{a}$ is abelian, we have $\mathfrak{z} = \mathfrak{a} \oplus \mathfrak{z}(\mathfrak{m})$ ($\mathfrak{a}$ plus the center of the Lie algebra $\mathfrak{m}$) and $\mathfrak{d} = [\mathfrak{m}, \mathfrak{m}]$. In other words:
\begin{equation}
\label{eq:l_complete_decomposition}
\mathfrak{l} =
\lefteqn{\overbrace{\phantom{
  \mathfrak{a} \oplus \mathfrak{z}(\mathfrak{m})
}}^{\displaystyle \mathfrak{z}}}
  \mathfrak{a} \oplus \underbrace{
  \mathfrak{z}(\mathfrak{m}) \oplus \mathfrak{d}
}_{\displaystyle \mathfrak{m}}.
\end{equation}
So the following Proposition, and in fact every single statement in the rest of the paper, would still be true if we substituted, respectively, $\mathfrak{a}$ and $\mathfrak{m}$ for $\mathfrak{z}$ and $\mathfrak{d}$. The advantage of introducing $\mathfrak{z}$ and $\mathfrak{d}$ is that the Margulis invariants (see below) live in a larger space ($\mathfrak{z}$ instead of $\mathfrak{a}$), and so are finer invariants. Maybe this could be helpful for further study.
\end{remark}

In fact it is possible to show (see \cite{Kna96}, Theorem 7.53b and c) that $Z$ meets every connected component of $L$. Thus we may also write
\[L = ZD.\]
By definition, $L$ acts trivially on $\mathfrak{z}$ and $Z$ acts trivially on $\mathfrak{l}$; the only nontrivial action is that of $D$ on $\mathfrak{d}$. Moreover, $D$ preserves the Killing form, which is negative definite on $\mathfrak{d}$ (since $\mathfrak{d} \subset \mathfrak{m} \subset \mathfrak{k}$). To sum everything up:

\begin{proposition}
\label{quasi-translation}
Any quasi-translation is an element of $(\Orth(\mathfrak{d}) \ltimes \mathfrak{d}) \times \mathfrak{z}$.
\end{proposition}
In other words, quasi-translations correspond to affine isometries of $\mathfrak{l}_{\Aff} = \hat{\mathfrak{l}} \cap \mathfrak{g}_{\Aff}$ that preserve the directions of $\mathfrak{d}$ and $\mathfrak{z}$ and act only by translation on the $\mathfrak{z}$ component.

\begin{example}\mbox{ }
\begin{enumerate}
\item For $G = \PSL_n(\mathbb{R})$, since the algebra $\mathfrak{l} = \mathfrak{a}$ is abelian, $\mathfrak{z}$ coincides with $\mathfrak{l}$ or equivalently $\mathfrak{d}$ is trivial, so a quasi-translation is simply a translation.
\item Take $G = \SO^+(4, 1)$; in this case we have $\mathfrak{d} = \mathfrak{m} \simeq \mathfrak{so}(3)$. This is the simplest example that requires the full strength of the proofs given in this paper.
\item Take $G = \PSU(3,1) \simeq \PSO^*(6)$: then $\mathfrak{m} \simeq \mathfrak{su}(2) \oplus \mathbb{R}$ as a Lie algebra, so that $\mathfrak{z}(\mathfrak{m}) \simeq \mathbb{R}$ and $\mathfrak{d} \simeq \mathfrak{su}(2)$. This shows that all three spaces in the decomposition \eqref{eq:l_complete_decomposition} can be nonzero, even when $G$ is simple.
\end{enumerate}
A table giving the algebra $\mathfrak{m}$ for every simple algebra $\mathfrak{g}$ may be found in \cite{Kna96}, Appendix~C.
\end{example}

\subsection{Canonical identifications}
\label{sec:canonical}

Here we introduce canonical identifications (up to quasi-translation) between different spaces $A^\sube_g$ (Corollary \ref{canonical_identification}), and use them to define the Margulis invariant of an $\mathbb{R}$-regular map. We also check that these identifications commute with certain "natural" projections (Lemma \ref{projections_commute}).

The following two properties are immediate consequences of Claims \ref{Asubge_is_ampa}, \ref{pair_transitivity} and \ref{u_l_invariant_d_fixed}:
\begin{corollary}
\label{canonical_identification}
Let $(\hat{\mathfrak{p}}_1, \hat{\mathfrak{p}}_2)$ be a pair of transverse affine m.p.a.'s. Then any map $\phi \in G \ltimes \mathfrak{g}$ such that $\phi(\hat{\mathfrak{p}}_1, \hat{\mathfrak{p}}_2) = (\hat{\mathfrak{p}}^+, \hat{\mathfrak{p}}^-)$ gives, by restriction, an identification of the intersection $\hat{\mathfrak{p}}_1 \cap \hat{\mathfrak{p}}_2$ with $\hat{\mathfrak{l}}$, which is unique up to composition on the left by a quasi-translation.
\end{corollary}
Here by $\phi(\hat{\mathfrak{p}}_1, \hat{\mathfrak{p}}_2)$ we mean the pair $(\phi(\hat{\mathfrak{p}}_1), \phi(\hat{\mathfrak{p}}_2))$. Note that if $\hat{\mathfrak{p}}_1 \cap \hat{\mathfrak{p}}_2$ is obtained in another way as an intersection of two affine m.p.a.'s, the identification with $\hat{\mathfrak{l}}$ may differ not just by a quasi-translation, but also by an element of the Weyl group.
\begin{corollary}
\label{V=_translation}
Let $g \in G \ltimes \mathfrak{g}$ be an $\mathbb{R}$-regular map. Let $\phi \in G \ltimes \mathfrak{g}$ be any map such that $\phi(A^\subge_g, A^\suble_g) = (\hat{\mathfrak{p}}^+, \hat{\mathfrak{p}}^-)$. Then the restriction of the conjugate $\phi g \phi^{-1}$ to $\hat{\mathfrak{l}}$ is a quasi-translation.
\end{corollary}
This leads to the following proposition. We call $\pi_\mathfrak{z}$ the projection from $\mathfrak{l}$ onto $\mathfrak{z}$ parallel to $\mathfrak{d}$.
\begin{proposition}
\label{margulis_invariant_prop}
Let $g \in G \ltimes \mathfrak{g}$ be an $\mathbb{R}$-regular map. Take any point $x$ in the affine space $A^{\sube}_{g} \cap \mathfrak{g}_{\Aff}$ and any map $\phi \in G$ such that $\phi(V^{\subge}_{g}, V^{\suble}_{g}) = (\mathfrak{p}^+, \mathfrak{p}^-)$. Then the vector
\[M(g) := \pi_\mathfrak{z}(\phi(g(x)-x)) \in \mathfrak{z}\]
does not depend on the choice of $x$ or $\phi$.
\end{proposition}
\begin{definition}
\label{margulis_invariant_def}
The vector $M(g)$ is called the \emph{Margulis invariant} of $g$.
\end{definition}

The proposition is more or less an immediate consequence of all the previous statements; but since the Margulis invariant is the central object of this paper, we give the detailed proof.

\begin{proof}[Proof of Proposition \ref{margulis_invariant_prop}.]~
\begin{itemize}
\item Let us first check that $M(g) \in \mathfrak{z}$. By hypothesis we have $x \in A^{\sube}_{g}$, hence also $g(x) \in A^{\sube}_{g}$. On the other hand, since $\mathfrak{g}$ is an extended affine map, it stabilizes the affine space $\mathfrak{g}_{\Aff}$; and the difference between two elements of $\mathfrak{g}_{\Aff}$ (two affine points) is an element of $\mathfrak{g}$ (a vector). It follows that
\[g(x) - x \in A^{\sube}_{g} \cap \mathfrak{g} = V^{\sube}_{g}.\]
Using the definition of $\phi$ and a purely linear version of Corollary \ref{dynamical_spaces_description}, we then have
\[\phi(g(x) - x) \in \mathfrak{l},\]
hence
\[\pi_{\mathfrak{z}}(\phi(g(x) - x)) \in \mathfrak{z}.\]

\item The independence of the result on the choice of $\phi$ essentially follows from a linear version of Corollary \ref{canonical_identification}. Indeed, let $\phi'$ be another element of $G$ satisfying the hypothesis; then a linear version of Corollary \ref{canonical_identification} says that in restriction to $\mathfrak{l}$, we have $\phi' = l \phi$ where $l$ is some quasi-translation that is also an element of $G$. Now by Proposition \ref{quasi-translation}, a quasi-translation without translation part is just an element of $\Orth(\mathfrak{d})$, and acts trivially on $\mathfrak{z}$. It follows that $\phi(g(x) - x)$ and $l(\phi(g(x) - x))$ have the same $\mathfrak{z}$-component.

\item The independence on the choice of $x$ is a consequence of Corollary \ref{V=_translation}. Indeed, let $\hat{\phi}$ be an element of $G \ltimes \mathfrak{g}$ whose restriction to $\mathfrak{g}$ (linear part) is equal to $\phi$ and such that
\[\hat{\phi}(A^{\subge}_{g}, A^{\suble}_{g}) = (\hat{\mathfrak{p}}^+, \hat{\mathfrak{p}}^-).\]
Then we may rewrite
\begin{equation}
\label{eq:margulis_invariant_alternative}
M(g) = \pi_{\mathfrak{z}}(g'(x') - x'),
\end{equation}
where we set $g' := \hat{\phi} g \hat{\phi}^{-1}$ and $x' = \hat{\phi}(x)$. By Corollary \ref{dynamical_spaces_description}, $\hat{\phi}$ induces a bijection between the extended affine spaces $A^{\sube}_g$ and $\hat{\mathfrak{l}}$, hence between the actual affine spaces $A^{\sube}_g \cap \mathfrak{g}_{\Aff}$ and $\hat{\mathfrak{l}} \cap \mathfrak{g}_{\Aff}$; so now our task is to show that the formula \eqref{eq:margulis_invariant_alternative} gives the same result for every choice of $x' \in \hat{\mathfrak{l}} \cap \mathfrak{g}_{\Aff}$.

Now by Corollary \ref{V=_translation}, $g'$ is a quasi-translation. By Proposition \ref{quasi-translation}, it follows that $g'$ acts only by translation on the $\mathfrak{z}$-component. In other words, let $\hat{\pi}_{\mathfrak{z}}$ be an affine version of $\pi_{\mathfrak{z}}$, defined as the projection from $\hat{\mathfrak{l}}$ onto $\hat{\mathfrak{z}} := \mathfrak{z} \oplus \mathbb{R}_0$ parallel to $\mathfrak{d}$; then we have
\[\hat{\pi}_{\mathfrak{z}} \circ g' = \tau_v \circ \hat{\pi}_{\mathfrak{z}}\]
for some vector $v \in \mathfrak{z}$, so that
\[\pi_{\mathfrak{z}}(g'(x')-x') = \hat{\pi}_{\mathfrak{z}}(g'(x')-x') = \tau_v(\hat{\pi}_{\mathfrak{z}}(x')) - \hat{\pi}_{\mathfrak{z}}(x').\]
Here $\hat{\pi}_{\mathfrak{z}}(x')$ is an element of the actual affine space $\hat{\mathfrak{z}} \cap \mathfrak{g}_{\Aff}$. It follows that
\[\tau_v(\hat{\pi}_{\mathfrak{z}}(x')) - \hat{\pi}_{\mathfrak{z}}(x') = v,\]
and the vector $M(g) = v$ does not depend on the choice of $x'$ (or $x$). \qedhere
\end{itemize}
\end{proof}

\begin{lemma}
\label{projections_commute}
Take any affine m.p.a. $\hat{\mathfrak{p}}_1$. Let $\mathfrak{n}_1$ be the nilradical of its linear part, and $\hat{\mathfrak{p}}_2$ and $\hat{\mathfrak{p}}'_2$ be any two affine m.p.a.'s both transverse to $\hat{\mathfrak{p}}_1$. Let $\phi$ (resp. $\phi'$) be an element of $G \ltimes \mathfrak{g}$ that sends the pair $(\hat{\mathfrak{p}}_1, \hat{\mathfrak{p}}_2)$ (resp. $(\hat{\mathfrak{p}}_1, \hat{\mathfrak{p}}'_2)$) to $(\hat{\mathfrak{p}}^+, \hat{\mathfrak{p}}^-)$. Let
\[\psi: \hat{\mathfrak{p}}_1 \cap \hat{\mathfrak{p}}_2 \longlongrightarrow \hat{\mathfrak{p}}_1 \cap \hat{\mathfrak{p}}'_2\]
be the projection parallel to $\mathfrak{n}_1$. Then the map $\overline{\psi}$ defined by the commutative diagram
\[
\begin{tikzcd}
  \hat{\mathfrak{l}}
    \arrow{rr}{\overline{\psi}}
&
& \hat{\mathfrak{l}}\\
\\
  \hat{\mathfrak{p}}_1 \cap \hat{\mathfrak{p}}_2
    \arrow{rr}{\psi}
    \arrow{uu}{\phi}
&
& \hat{\mathfrak{p}}_1 \cap \hat{\mathfrak{p}}'_2
    \arrow{uu}{\phi'}
\end{tikzcd}
\]
is a quasi-translation.
\end{lemma}
The maps $\phi$ and $\phi'$ exist by Claim \ref{pair_transitivity}, and their restrictions that appear in the diagram are unique up to quasi-translation by Corollary \ref{canonical_identification}. The projection $\psi$ is well-defined because $\hat{\mathfrak{p}}^+ = \mathfrak{n}^+ \oplus \hat{\mathfrak{l}} = \mathfrak{n}^+ \oplus (\hat{\mathfrak{p}}^+ \cap \hat{\mathfrak{p}}^-)$, and so $\hat{\mathfrak{p}}_1 = \phi'^{-1}(\hat{\mathfrak{p}}^+) = \mathfrak{n}_1 \oplus (\hat{\mathfrak{p}}_1 \cap \hat{\mathfrak{p}}'_2)$.
\begin{proof}
Without loss of generality, we may assume that $\phi = \Id$ (otherwise we simply replace the three affine m.p.a.'s by their images under $\phi^{-1}$.) Then we have $\hat{\mathfrak{p}}_1 = \hat{\mathfrak{p}}^+$, $\hat{\mathfrak{p}}_2 = \hat{\mathfrak{p}}^-$ and $\hat{\mathfrak{p}}'_2 = \phi'^{-1}(\hat{\mathfrak{p}}^-)$, where $\phi'$ can be any map stabilizing the space $\hat{\mathfrak{p}}^+$. We want to show that the map $\phi' \circ \psi$ is a quasi-translation.

We know that $\phi'$ lies in the stabilizer $N_{G \ltimes \mathfrak{g}}(\hat{\mathfrak{p}}^+)$, which is equal to \mbox{$P^+ \ltimes \mathfrak{p}^+$}, where $P^+ := N_G(\mathfrak{p}^+)$ is the minimal parabolic subgroup with Lie algebra $\mathfrak{p}^+$. We shall use the Langlands decomposition
\[P^+ = MAN^+ = LN^+,\]
where $N^+$ is the connected group with Lie algebra $\mathfrak{n}^+$ (see \eg \cite{Kna96}, Proposition 7.83). Since $L$ normalizes $\mathfrak{n}^+$ and $\mathfrak{l} + \mathfrak{n}^+ = \mathfrak{p}^+$, this generalizes to the "affine Langlands decomposition"
\[P^+ \ltimes \mathfrak{p}^+ = (L \ltimes \mathfrak{l})(N^+ \ltimes \mathfrak{n}^+).\]
Thus we may write $\phi' = l \circ n$ with $l \in L \ltimes \mathfrak{l}$ and $n \in N^+ \ltimes \mathfrak{n}^+$.

We shall use the following fact: every element $n$ of the group \mbox{$N^+ \ltimes \mathfrak{n}^+$} stabilizes the space $\mathfrak{n}^+$ and induces the identity map on the quotient space $\hat{\mathfrak{p}}^+ / \mathfrak{n}^+$. Indeed, when $n$ lies in the "linear" group $N^+$, since $N^+$ is connected, this follows from the fact that $\mathfrak{n}^+$ is an ideal of $\mathfrak{p}^+$. When $n$ is a pure translation by a vector of $\mathfrak{n}^+$, this is obvious.

By definition, $\psi$ also stabilizes $\mathfrak{n}^+$ and induces the identity on $\hat{\mathfrak{p}}^+ / \mathfrak{n}^+$; hence so does the map $n \circ \psi$. But we also know that $n \circ \psi$ is defined on $\hat{\mathfrak{p}}_1 \cap \hat{\mathfrak{p}}_2 = \hat{\mathfrak{l}}$, and sends it onto
\[n \circ \psi(\hat{\mathfrak{p}}_1 \cap \hat{\mathfrak{p}}_2) = n(\hat{\mathfrak{p}}_1 \cap \hat{\mathfrak{p}}'_2) = l^{-1}(\hat{\mathfrak{l}}) = \hat{\mathfrak{l}}.\]
Hence the map $n \circ \psi$ is the identity on $\hat{\mathfrak{l}}$. It follows that $\overline{\psi} = \phi' \circ \psi = {l \circ n \circ \psi = l}$ (in restriction to $\hat{\mathfrak{l}}$), hence $\overline{\psi}$ is a quasi-translation as required.
\end{proof}

\subsection{Metric properties}
\label{sec:metric}

Here we introduce some conventions and define two important metric properties of $\mathbb{R}$-regular maps: $C$-non-degeneracy (which means that the geometry of the map is not too close to a degenerate case), and contraction strength.

We introduce on $\hat{\mathfrak{g}}$ a Euclidean norm such that the subspaces $\mathfrak{n}^+$, $\mathfrak{n}^-$, $\mathfrak{d}$, $\mathfrak{z}$ and $\mathbb{R}_0$ are pairwise orthogonal, and whose restriction to $\mathfrak{d}$ agrees with the Killing form up to sign (indeed the latter is negative definite on $\mathfrak{d} \subset \mathfrak{m} \subset \mathfrak{k}$). For any linear map $g$ acting on $\hat{\mathfrak{g}}$, we write $\|g\| := \sup_{x \neq 0} \frac{\|g(x)\|}{\|x\|}$ its operator norm.

Consider a Euclidean space $E$ (for the moment, the reader may suppose that $E = \hat{\mathfrak{g}}$; later we will also need the case $E = \ext^p \hat{\mathfrak{g}}$ for some integer $p$). We introduce on the projective space $\mathbb{P}(E)$ a metric by setting, for every $\overline{x}, \overline{y} \in \mathbb{P}(E)$,
\[\alpha (\overline{x}, \overline{y}) := \arccos \frac{| \langle x, y \rangle |}{\|x\| \|y\|} \in \textstyle [0, \frac{\pi}{2}],\]
where $x$ and $y$ are any vectors representing respectively $\overline{x}$ and $\overline{y}$ (obviously, the value does not depend on the choice of $x$ and $y$). This measures the angle between the lines $\overline{x}$ and $\overline{y}$. For shortness' sake, we will usually simply write $\alpha(x, y)$ with $x$ and $y$ some actual vectors in $E \setminus \{0\}$.

For any vector subspace $F \subset E$ and any radius $\eps > 0$, we shall denote the $\eps$-neighborhood of $F$ in $\mathbb{P}(E)$ by:
\[B_{\mathbb{P}}(F, \eps) := \setsuch{x \in \mathbb{P}(E)}{\alpha(x,\mathbb{P}(F)) < \eps}.\]
(You may think of it as a kind of "conical neighborhood".)

Consider a metric space $(\mathcal{M}, \delta)$; let $X$ and $Y$ be two subsets of $\mathcal{M}$. We shall denote the ordinary, minimum distance between $X$ and $Y$ by
\[\delta(X, Y) := \inf_{x \in X} \inf_{y \in Y} \delta(x, y),\]
as opposed to the Hausdorff distance, which we shall denote by
\[\delta^\mathrm{Haus}(X, Y) := \max\left( \sup_{x \in X} \delta\big(\{x\}, Y\big),\; \sup_{y \in Y} \delta\big(\{y\}, X\big) \right).\]

Finally, we introduce the following notation. Let $A$ and $B$ be two positive quantities, and $p_1, \ldots, p_k$ some parameters. Whenever we write
\[A \lesssim_{p_1, \ldots, p_k} B,\]
we mean that there is a constant $K$, depending on nothing but $p_1, \ldots, p_k$, such that $A \leq KB$. (If we do not write any subscripts, this means of course that $K$ is an "absolute" constant --- or at least, that it does not depend on any "local" parameters; we consider the "global" parameters such as the choice of $G$ and of the Euclidean norms to be fixed once and for all.) Whenever we write
\[A \asymp_{p_1, \ldots, p_k} B,\]
we mean that $A \lesssim_{p_1, \ldots, p_k} B$ and $B \lesssim_{p_1, \ldots, p_k} A$ at the same time.

\begin{definition}
\label{regular_definition}
Take a pair of affine m.p.a.'s $(\hat{\mathfrak{p}}_1, \hat{\mathfrak{p}}_2)$. An \emph{optimal canonizing map} for this pair is a map $\phi \in G \ltimes \mathfrak{g}$ satisfying
\[\phi(\hat{\mathfrak{p}}_1, \hat{\mathfrak{p}}_2) = (\hat{\mathfrak{p}}^+, \hat{\mathfrak{p}}^-)\]
and minimizing the quantity $\max \left( \|\phi\|, \|\phi^{-1}\| \right)$. By Claim \ref{pair_transitivity} and a compactness argument, such a map exists iff $\hat{\mathfrak{p}}_1$ and $\hat{\mathfrak{p}}_2$ are transverse.

We define an \emph{optimal canonizing map} for an $\mathbb{R}$-regular map $g \in G \ltimes \mathfrak{g}$ to be an optimal canonizing map for the pair $(A^\subge_g, A^\suble_g)$.

Let $C \geq 1$. We say that a pair of affine m.p.a.'s $(\hat{\mathfrak{p}}_1, \hat{\mathfrak{p}}_2)$ (resp. an $\mathbb{R}$\nobreakdash-\hspace{0pt}regular map $g$) is \emph{$C$-non-degenerate} if it has an optimal canonizing map $\phi$ such that $\left \|\phi^{\pm 1} \right\| \leq C$.

Now take $g_1$, $g_2$ two $\mathbb{R}$-regular maps in $G \ltimes \mathfrak{g}$. We say that the pair $(g_1, g_2)$ is \emph{$C$-non-degenerate} if every one of the four possible pairs $(A^{\subge}_{g_i}, A^{\suble}_{g_j})$ is $C$\nobreakdash-\hspace{0pt}non-degenerate.
\end{definition}

The point of this definition is that there are a lot of calculations in which, when we treat a $C$-non-degenerate pair of spaces as if they were perpendicular, we err by no more than a (multiplicative) constant depending on~$C$. The following result will often be useful:

\begin{lemma}
\label{bounded_norm_is_bilipschitz}
Let $C \geq 1$. Then any map $\phi \in \GL(E)$ such that $\|\phi^{\pm 1}\| \leq C$ induces a $C^2$-Lipschitz continuous map on $\mathbb{P}(E)$.
\end{lemma}
\begin{proof}
It is sufficient to check this for the restriction of $\phi$ to every 2-dimensional subspace of $E$. But in 2-dimensional space, using singular value decomposition (see the proof of Lemma \ref{regular_to_proximal} (iii) for a definition), this identity is straightforward. In fact it turns out that the Lipschitz constant of $\phi$ acting on $\mathbb{P}(E)$ is exactly $\|\phi\| \|\phi^{-1}\|$.
\end{proof}

\begin{remark}
The set of transverse pairs of extended affine spaces is characterized by two open conditions: there is of course transversality of the spaces, but also the requirement that each space not be contained in $\mathfrak{g}$. What we mean here by "degeneracy" is failure of one of these two conditions. Thus the property of a pair $(\hat{\mathfrak{p}}_1, \hat{\mathfrak{p}}_2)$ being $C$-non-degenerate actually encompasses two properties.

First, it implies that the spaces $\hat{\mathfrak{p}}_1$ and $\hat{\mathfrak{p}}_2$ are transversal in a quantitative way. More precisely, this means that some continuous function that would vanish if the spaces were not transversal is bounded below. An example of such a function is the smallest non identically vanishing of the "principal angles" defined in the proof of Lemma \ref{regular_to_proximal} (iv).

Second, it implies that both $\hat{\mathfrak{p}}_1$ and $\hat{\mathfrak{p}}_2$ are "not too close" to the space~$\mathfrak{g}$ (in the same sense). In purely affine terms, this means that the affine spaces $\hat{\mathfrak{p}}_1 \cap \mathfrak{g}_{\Aff}$ and $\hat{\mathfrak{p}}_2 \cap \mathfrak{g}_{\Aff}$ contain points that are not too far from the origin.

Both conditions are necessary, and appeared in the previous literature (such as \cite{Mar87} and \cite{AMS02}); but so far, they have always been treated separately.
\end{remark}

\begin{definition}
\label{s_definition}
Let $s > 0$. For an $\mathbb{R}$-regular map $g \in G \ltimes \mathfrak{g}$, we say that $g$ is \emph{$s$-contracting} if we have:
\[\forall (x, y) \in V^{\subl}_{g} \times A^{\subge}_{g}, \quad
  \frac{\|g(x)\|}{\|x\|} \leq s\frac{\|g(y)\|}{\|y\|}.\]
(Note that by Corollary \ref{dynamical_spaces_description} the spaces $V^{\subl}_{g}$ and $A^{\subge}_{g}$ always have the same dimensions as $\mathfrak{n}^-$ and $\hat{\mathfrak{p}}^+$ respectively, hence they are nonzero.)

We define the \emph{strength of contraction} of $g$ to be the smallest number $s(g)$ such that $g$ is $s(g)$-contracting. In other words, we have
\[s(g) =
\left\| \restr{g}     {V^{\subl}_{g}}  \right\|
\left\| \restr{g^{-1}}{A^{\subge}_{g}} \right\|.\]
\end{definition}

In yet other words, $s(g)$ is the \emph{inverse} of the "singular value gap" between $V^\subl_g$ and~$A^\subge_g$ (see the proof of Lemma \ref{regular_to_proximal} (iii) for the definition of singular values). We chose the convention where a "strongly contracting" map has a \emph{small} value of $s$.

\begin{remark}
\label{contraction_strength_and_inverse}
Even though we will not use it, it is useful to keep in mind the following property. One can show that if $g$ is $C$-non-degenerate with $s(g) \leq 1$, we actually have $s(g^{-1}) \asymp_C s(g)$. Thus the apparent lack of symmetry in the definition (why take $V^{\subl}_{g}$ and~$A^{\subge}_{g}$ rather than $A^{\suble}_{g}$ and~$V^{\subg}_{g}$?) is not a real problem.
\end{remark}

\begin{remark}
\label{contraction_strength_grows}
Note that for any $\mathbb{R}$-regular map $g \in G \ltimes \mathfrak{g}$, we have
\[\log s(g^N) \underset{N \to \infty}{\sim} -N \log \rho,\]
where $\rho$ is the spectral gap of $g$ between $V^{\subl}_{g}$ and $A^{\subge}_{g}$. By definition, $\rho > 0$; it follows that
\[s(g^N) \underset{N \to \infty}{\to} 0.\]
\end{remark}

\subsection{Comparison of metric properties in the affine and linear case}
\label{sec:affine_to_linear}

For any map $f \in G \ltimes \mathfrak{g}$, we denote by $\ell(f)$ the linear part of $f$, seen as an element of $G \ltimes \mathfrak{g}$ by identifying $G$ with the stabilizer of the "origin" $\mathbb{R}_0$. In other words, for every $(x, t) \in \mathfrak{g} \oplus \mathbb{R}_0 = \hat{\mathfrak{g}}$, we set
\[\ell(f)(x, t) = f(x, 0) + (0, t).\]
(Seeing $G$ as a subgroup of $G \ltimes \mathfrak{g}$ allows us to avoid introducing new definitions of $C$-non-degeneracy and contraction strength for elements of $G$.)
\begin{lemma}
\label{affine_to_vector}
Let $C \geq 1$, and take any $C$-non-degenerate $\mathbb{R}$-regular map $g$ (or pair of maps $(g, h)$) in $G \ltimes \mathfrak{g}$. Then:
\begin{enumerate}[(i)]
\item The map $\ell(g)$ (resp. the pair $(\ell(g), \ell(h))$) is still $C$-non-degenerate;
\item We have $s(\ell(g)) \leq s(g)$;
\item Suppose that $s(g^{-1}) \leq 1$. Then we actually have
\[s(g) \asymp_C s(\ell(g)) \left\| \restr{g}{A^\sube_g} \right\|.\]
\end{enumerate}
\end{lemma}
The proof of the first two points is just a formal verification, and contains no surprises.
\begin{proof} \mbox{ }
\begin{enumerate}[(i)]
\item We will show the result only for one map $g$; for a pair of maps the reasoning is analogous.

Let $\phi$ be some optimal canonizing map for $g$. Then clearly $A^{\subge}_{\ell(g)} = V^{\subge}_{g} \oplus \mathbb{R}_0$, and
\begin{align*}
\ell(\phi)(A^{\subge}_{\ell(g)})
  &= \ell(\phi)(V^{\subge}_{g} \oplus \mathbb{R}_0) \\
  &= \phi(V^{\subge}_{g}) \oplus \mathbb{R}_0 \\
  &= \hat{\mathfrak{p}}^+;
\end{align*}
similarly, $\ell(\phi)(A^{\suble}_{\ell(g)}) = \hat{\mathfrak{p}}^-$. Thus $\ell(\phi)$ is a canonizing map for $\ell(g)$. On the other hand, we have
\begin{align*}
\left \|{\ell(\phi)} \right\|
  &= \max \left( \left \|{\restr{\phi}{\mathfrak{g}}} \right\| , 1 \right) \\
  &\leq \max \left( \left \|{\phi} \right\| , 1 \right) \\
  &\leq \max(C, 1),
\end{align*}
and similarly for $\phi^{-1}$. As $C \geq 1$, we get that $\ell(g)$ is $C$\hyph{}non\hyph{}degenerate.

\item We have:
\begin{align*}
s(\ell(g))
  &= \left\| \restr{\ell(g)}{V^{\subl}_{\ell(g)}}  \right\|
     \left\| \restr{\ell(g)^{-1}}{A^{\subge}_{\ell(g)}} \right\| \\
  &= \left\| \restr{g}{V^{\subl}_{g}}  \right\|
     \max\left( \left\| \restr{g^{-1}}{V^{\subge}_{g}} \right\|, 1 \right) \\
  &\leq \left\| \restr{g}{V^{\subl}_{g}}  \right\|
        \max\left( \left\| \restr{g^{-1}}{A^{\subge}_{g}} \right\|, 1 \right) \\
  &= s(g).
\end{align*}
To justify the last equality, note that $V^\sube_g \subset A^{\subge}_{g}$ is nonzero by Corollary \ref{dynamical_spaces_description}, and that all eigenvalues of $g^{-1}$ restricted to the former subspace have modulus~1, hence $\left\| \restr{g^{-1}}{A^{\subge}_{g}} \right\|  \geq \left\| \restr{g^{-1}}{V^{\sube}_{g}} \right\| \geq 1$.

\item We have, by definition:
\[s(g) =
    \left\| \restr{g}{V^{\subl}_{g}}  \right\|
    \left\| \restr{g^{-1}}{A^{\subge}_{g}} \right\|.\]
Let $\phi$ be an optimal canonizing map for $g$. Since $g$ is $C$\hyph{}non\hyph{}degenerate (and $\phi(A^{\sube}_{g}) = \hat{\mathfrak{l}}$ is orthogonal to $\phi(V^{\subg}_{g}) = \mathfrak{n}^+$, the latter equality following from Corollary~\ref{dynamical_spaces_description}), it follows that
\[s(g) \asymp_C
    \left\| \restr{g}     {V^{\subl}_{g}}  \right\|
    \max \left(
    \left\| \restr{g^{-1}}{A^{\sube}_{g}}  \right\|,
    \left\| \restr{g^{-1}}{V^{\subg}_{g}}  \right\|
    \right).\]
Clearly we have $\left\| \restr{g^{-1}}{A^{\sube}_{g}} \right\| \geq \left\| \restr{g^{-1}}{V^{\sube}_{g}} \right\| \geq 1$ (see previous point). On the other hand, since $s(g^{-1}) \leq 1$, we have $\left\| \restr{g^{-1}}{V^{\subg}_{g}}  \right\| \leq 1$. It follows that
\[s(g) \asymp_C
    \left\| \restr{g}     {V^{\subl}_{g}} \right\|
    \left\| \restr{g^{-1}}{A^{\sube}_{g}} \right\|.\]
By Corollary \ref{V=_translation}, the conjugate of $\restr{g}{A^{\sube}_{g}}$ by $\phi$ is a quasi-translation. By Proposition \ref{quasi-translation} characterizing quasi-translations, we may write
\[\phi \restr{g}{A^{\sube}_{g}} \phi^{-1} = \tau_v \rho,\]
where $\rho$ is an orthogonal automorphism of the subspace $\mathfrak{d}$, and $\tau_v$ is the translation by some vector $v \in \mathfrak{l}$. Since $\rho$ preserves the Euclidean norm (it preserves the Killing form, and by convention they agree on $\mathfrak{d}$), it has no influence on the operator norm; and clearly $\|\tau_v\| = \|\tau_{-v}\|$. It follows that $\|\rho^{-1} \tau_v^{-1}\| = \|\tau_{-v}\| = \|\tau_v\| = \|\tau_v \rho\|$, hence $\left\| \restr{g^{-1}}{A^{\sube}_{g}} \right\| \asymp_C \left\| \restr{g}{A^{\sube}_{g}} \right\|$. Thus we get
\[s(g) \asymp_C
    \left\| \restr{g}{V^{\subl}_{g}} \right\|
    \left\| \restr{g}{A^{\sube}_{g}} \right\|.\]
A similar estimate holds for $\ell(g)$; but since $\ell(g)$ restricted to $A^{\sube}_{\ell(g)}$ has no translation part, the second factor disappears:
\[s(\ell(g)) \asymp_C
    \left\| \restr{\ell(g)}{V^{\subl}_{g}} \right\|.\]
Since $g$ and $\ell(g)$ coincide on $V^{\subl}_{g}$, we conclude that
\[s(g) \asymp_C
    s(\ell(g))\left\| \restr{g}{A^{\sube}_{g}} \right\|\]
as required. \qedhere
\end{enumerate}
\end{proof}

\section{$\mathbb{R}$-regularity of products}
\label{sec:regular_product}

The goal of this section is to prove Proposition \ref{regular_product}, which essentially states in a quantitative way that under some conditions, the product of two $\mathbb{R}$-regular maps is still $\mathbb{R}$-regular.

\subsection{Proximal case}
\label{sec:proximal_product}
Let $E$ be a Euclidean space. (In practice, we will apply the results of this subsection to $E = \ext^p \hat{\mathfrak{g}}$ for some integer $p$.)

Our first goal is to show Proposition \ref{proximal_product}, which is analogous to Proposition \ref{regular_product} (and will be used to prove it), but with proximal maps instead of $\mathbb{R}$-regular ones. We begin with a few definitions.

\begin{definition}
\label{proximal_definition}
Let $\gamma \in \GL(E)$. Let $\lambda$ be an eigenvalue of $\gamma$ with maximal modulus. We say that $\gamma$ is \emph{proximal} if $\lambda$ is unique and has multiplicity~1. We may then decompose~$E$ into a direct sum of a line $E^s_\gamma$, called its \emph{attracting space}, and a hyperplane $E^u_\gamma$, called its \emph{repelling space}, both stable by $\gamma$ and such that:
\[\begin{cases}
\restr{\gamma}{E^s_\gamma} = \lambda \Id \\
\text{for every eigenvalue } \mu \text{ of } \restr{\gamma}{E^u_\gamma},\; |\mu| < |\lambda|.
\end{cases}\]
\end{definition}

\begin{definition}
Consider a line $E^s$ and a hyperplane $E^u$ of $E$, transverse to each other. An \emph{optimal canonizing map} for the pair $(E^s, E^u)$ is a map $\phi \in GL(E)$ satisfying
\[\phi(E^s) \perp \phi(E^u)\]
and minimizing the quantity $\max \left( \|\phi\|, \|\phi^{-1}\| \right)$.

We define an \emph{optimal canonizing map} for a proximal map $\gamma \in \GL(E)$ to be an optimal canonizing map for the pair $(E^s_\gamma, E^u_\gamma)$.

Let $C \geq 1$. We say that the pair formed by a line and a hyperplane $(E^s, E^u)$ (resp. that a proximal map $\gamma$) is \emph{$C$-non-degenerate} if it has an optimal canonizing map $\phi$ such that $\left \|\phi^{\pm 1} \right\| \leq C$. (The condition for a pair is equivalent to the condition that the angle between $E^s$ and $E^u$ is larger than or equal to $2\arctan(C^{-2})$.)

Now take $\gamma_1, \gamma_2$ two proximal maps in $\GL(E)$. We say that the pair $(\gamma_1, \gamma_2)$ is \emph{$C$-non-degenerate} if every one of the four possible pairs $(E^s_{\gamma_i}, E^u_{\gamma_j})$ is $C$-non-degenerate.
\end{definition}

\begin{definition}
\label{s_tilde_definition}
Let $\gamma \in \GL(E)$ be a proximal map. We define the \emph{strength of contraction} of $\gamma$ by
\[\tilde{s}(\gamma) := \frac{\left\| \restr{\gamma}{E^u_\gamma} \right\|}{|\lambda|};\]
we say that $\gamma$ is \emph{$\tilde{s}$-contracting} if $\tilde{s}(\gamma) \leq \tilde{s}$.
\end{definition}

Note that this definition is different from the one we used in the context of $\mathbb{R}$-regular maps (hence the new notation $\tilde{s}$).

\begin{proposition}
\label{proximal_product}
For every $C \geq 1$, there is a positive constant $\tilde{s}_1(C)$ with the following property. Take a $C$-non-degenerate pair of proximal maps $\gamma_1, \gamma_2$ in $\GL(E)$, and suppose that both $\gamma_1$ and $\gamma_2$ are $\tilde{s}_1(C)$-contracting. Then $\gamma_1 \gamma_2$ is proximal, and we have:
\begin{samepage}
\begin{enumerate}[(i)]
\item $\alpha \left(E^s_{\gamma_1 \gamma_2},\; E^s_{\gamma_1} \right) \lesssim_C \tilde{s}(\gamma_1)$;
\item $\tilde{s}(\gamma_1 \gamma_2) \lesssim_C \tilde{s}(\gamma_1)\tilde{s}(\gamma_2)$.
\end{enumerate}
\end{samepage}
\end{proposition}

Before proceeding, we need the following technical lemma, which says roughly that a proximal map $\gamma$ is strongly contracting in the sense of Definition \ref{s_tilde_definition} if and only if it is strongly Lipschitz-contracting on some subset of the projective space $\mathbb{P}(E)$.

For any set $X \subset \mathbb{P}(E)$, we introduce the following notation for the Lipschitz constant of $\gamma$ restricted to $X$:
\[\mathcal{L}(\gamma, X) :=
\sup_{\substack{(x, y) \in X^2 \\ x \neq y}}
\frac{\alpha (\gamma(x), \gamma(y))}{\alpha (x, y)}.\]

\begin{lemma}
\label{Lipschitz}
For any $C \geq 1$, $\zeta \in \;]0, \frac{\pi}{2}[$, for any proximal $C$-non-degenerate map $\gamma$, we have:
\begin{samepage}
\begin{subequations}
\label{eq:Lipschitz}
  \begin{equation}
  \label{eq:Lipschitz2}
    \mathcal{L} \left( \gamma,\; B_{\mathbb{P}} (E^s_\gamma, \zeta) \right)
\asymp_{C, \zeta} \tilde{s}(\gamma)
  \end{equation}
  \begin{equation}
  \label{eq:Lipschitz1}
    \mathcal{L} \left( \gamma,\; \mathbb{P}(E) \setminus B_{\mathbb{P}} (E^u_\gamma, \zeta) \right) \asymp_{C, \zeta} \tilde{s}(\gamma).
  \end{equation}
\end{subequations}
\end{samepage}
\end{lemma}

We shall actually only use the $\gtrsim$ part of \eqref{eq:Lipschitz2} and the $\lesssim$ part of \eqref{eq:Lipschitz1}. Note that clearly whenever $X \subset Y$ we have $\mathcal{L}(\gamma, X) \leq \mathcal{L}(\gamma, Y)$; hence for these two inequalities, it is sufficient to restrict our attention to small values of $\zeta$. The idea of this Lemma is that knowing just the Lipschitz constant of~$\gamma$ on a tiny neighborhood of its attracting space allows us to control $\tilde{s}(\gamma)$; but knowing $\tilde{s}(\gamma)$ actually allows us to control the Lipschitz constant of~$\gamma$ almost everywhere, except for a tiny neighborhood of its repelling space (that we have no hope to control).

\begin{proof}
Let $C \geq 1$, $\zeta \in \;]0, \frac{\pi}{2}[$. Consider a $C$-non-degenerate proximal map $\gamma$; let $\phi$ be an optimal canonizing map for $\gamma$. Then without loss of generality, we may replace $\gamma$ by~$\gamma' := \phi \gamma \phi^{-1}$. Indeed $\tilde{s}(\gamma) \asymp_C \tilde{s}(\gamma')$ is obvious. As for the other side, by Lemma \ref{bounded_norm_is_bilipschitz}, we have $\mathcal{L}(\phi, \mathbb{P}(E)) \asymp_C 1$, hence $\mathcal{L}(\gamma, X) \asymp_C \mathcal{L}(\gamma', \phi(X))$ for any set $X$. We also have
\begin{align*}
\phi \left( B_{\mathbb{P}} (E^s_\gamma,\; \zeta) \right)\; &\supset\; B_{\mathbb{P}} (E^s_{\gamma'},\; C^{-2} \zeta),\\
\phi \left( \mathbb{P}(E) \setminus B_{\mathbb{P}} (E^u_\gamma,\; \zeta) \right)\; &\subset\; \mathbb{P}(E) \setminus B_{\mathbb{P}} (E^u_{\gamma'},\; C^{-2} \zeta) ;
\end{align*}
and as remarked previously, $X \subset Y$ always implies $\mathcal{L}(\gamma, X) \leq \mathcal{L}(\gamma, Y)$.

It remains to show that for any $\zeta' \in \;]0, \frac{\pi}{2}[$, we have
\[\mathcal{L} \left( \gamma',\; B_{\mathbb{P}} (E^s_{\gamma'}, \zeta') \right)
\asymp_{\zeta'} \tilde{s}(\gamma')\]
(this implies \eqref{eq:Lipschitz2} by taking $\zeta' = C^{-2}\zeta$, and \eqref{eq:Lipschitz1} by taking $\zeta' > \frac{\pi}{2} - C^{-2}\zeta$).
Indeed, consider the projection
\[
\xymatrix@R=3pt{
  \pi_u: \mathbb{P}(E) \setminus \mathbb{P}(E^u_{\gamma'})
           \ar@{->}[r] &
         E^u_{\gamma'} \\
         \overline{x}
           \ar@{|->}[r] &
         \frac{x_u}{x_s},
}
\]
where $x_u$ and $x_s$ denote the components of $x$ in the decomposition $E = E^u_{\gamma'} \oplus E^s_{\gamma'}$ (and to make sense of division by $x_s$, we choose an isometrical identification of $E^s_{\gamma'}$ with $\mathbb{R}$). Since $E^s_{\gamma'}$ and $E^u_{\gamma'}$ are, by construction, orthogonal, it induces a homeomorphism from $B_{\mathbb{P}}(E^s_{\gamma'}, \zeta')$ to the ball $\setsuch{x \in E^u_{\gamma'}}{\|x\| \leq \tan \zeta'}$. A straightforward calculation shows that the said homeomorphism is bilipschitz, with a Lipschitz constant $K(\zeta')$ that does not at all depend on $\gamma$ or $C$. On the other hand, the Lipschitz constant of the conjugate map $\pi_u \gamma' \pi_u^{-1}$ (which is linear) is nothing other than $\tilde{s}(\gamma')$. Hence $\gamma'$ is Lipschitz-continuous with constant $K(\zeta')^2\tilde{s}(\gamma')$. The conclusion follows.
\end{proof}

\begin{proof}[Proof of Proposition \ref{proximal_product}]
Let $C \geq 1$, and let $(\gamma_1, \gamma_2)$ be a $C$-non-degenerate pair of $\tilde{s}_1(C)$-contracting proximal maps (for a value $\tilde{s}_1(C)$ to be specified later). Then by Lemma \ref{bounded_norm_is_bilipschitz}, for every $i$ and $j$ we have $\alpha(E^s_{\gamma_i}, E^u_{\gamma_j}) \geq \eta$ where we set $\eta := \frac{\pi}{2C^2}$.

An immediate corollary of Lemma \ref{Lipschitz} is that for every $C$-non-degenerate proximal map $\gamma$ and every $\zeta \leq \eta$, we have
\begin{equation}
\label{eq:eps6}
\gamma \left( \mathbb{P}(E) \setminus B_{\mathbb{P}}(E^u_\gamma, \zeta) \right) \subset B_{\mathbb{P}}\left( E^s_\gamma,\; K\left( C, \zeta \right)\tilde{s}(\gamma) \right)
\end{equation}
for some constant $K(C, \zeta)$. Indeed, $E^s_\gamma \in \mathbb{P}(E) \setminus B_{\mathbb{P}}(E^u_\gamma, \zeta)$ is a fixed point of $\gamma$ and $\mathrm{diam}(\mathbb{P}(E) \setminus B_{\mathbb{P}}(E^u_\gamma, \zeta)) \leq \frac{\pi}{2} \lesssim 1$.

For $i = 1, 2$, we introduce the numbers $\eta_i := K(C, \frac{\eta}{3})\tilde{s}(\gamma_i)$ and the sets
\[\begin{cases}
X_i^+ := B_{\mathbb{P}}(E^s_{\gamma_i}, \eta_i) \\
X_i^- := B_{\mathbb{P}}(E^u_{\gamma_i}, \frac{\eta}{3}).
\end{cases}\]
Then by \eqref{eq:eps6}, for every $i$ we have $\gamma_i(\mathbb{P}(E) \setminus X_i^-) \subset X_i^+$. Since $\tilde{s}(\gamma_i) \leq \tilde{s}_1(C)$, if we choose $\tilde{s}_1(C)$ small enough, we may suppose that $\eta_i \leq \frac{\eta}{3}$. Then these four sets are pairwise disjoint: for every $i$ and $j$, we have $X_i^+ \subset \mathbb{P}(E) \setminus X_j^-$. In particular, it follows that
\[\gamma_1 \gamma_2 \left( \mathbb{P}(E) \setminus X_2^- \right) \subset X_1^+.\]

Now by \eqref{eq:Lipschitz1}, we know that for every $i$
\begin{equation}
\label{eq:lip_less_s}
\mathcal{L} \left( \gamma_i,\; \mathbb{P}(E) \setminus X_i^- \right) \lesssim_{C} \tilde{s}(\gamma_i) \leq \tilde{s}_1(C).
\end{equation}
Once again, choosing $\tilde{s}_1(C)$ small enough, we may actually suppose that
\[\mathcal{L} \left( \gamma_i,\; \mathbb{P}(E) \setminus X_i^- \right) < 1.\]
Since $X_1^+ \subset \mathbb{P}(E) \setminus X_2^-$, it follows that $X_1^+$ is stable by $\gamma_1 \gamma_2$ and that
\[\mathcal{L} \left( \gamma_1 \gamma_2,\; X_1^+ \right) < 1.\]

We deduce from this that $\gamma_1 \gamma_2$ is proximal and $E^s_{\gamma_1 \gamma_2} \in X_1^+$ (see \cite{Tits72}, Lemma 3.8 for a proof), which settles the inequality (i). On the other hand, it is easy to see that $E^u_{\gamma_1 \gamma_2} \subset X_2^-$ (indeed, consider any point $x \in \mathbb{P}(E)$ belonging to $E^u_{\gamma_1 \gamma_2}$ but not to $X_2^-$: then we would have $\lim_{n \to \infty} (\gamma_1 \gamma_2)^n(x) = E^s_{\gamma_1 \gamma_2}$, which contradicts the fact that $E^u_{\gamma_1 \gamma_2}$ is a stable subspace). It follows that
\begin{align*}
\alpha(E^s_{\gamma_1 \gamma_2}, E^u_{\gamma_1 \gamma_2})
  &\geq \textstyle
    \alpha(E^s_{\gamma_1}, E^u_{\gamma_2}) - \eta_1 - \frac{\eta}{3}\\
  &\geq \textstyle
    \eta - \frac{\eta}{3} - \frac{\eta}{3}\\
  &= \textstyle \frac{\eta}{3}.
\end{align*}
Clearly, this implies that $\gamma_1 \gamma_2$ is $C'$-non-degenerate for some constant $C'$ that depends only on $\eta$, hence only on $C$.

This allows us to apply \eqref{eq:Lipschitz2} to $\gamma_1 \gamma_2$:
\[\textstyle
\tilde{s}(\gamma_1 \gamma_2) \lesssim_{C} \mathcal{L} \left( \gamma_1 \gamma_2,\; B_{\mathbb{P}}(E^s_{\gamma_1 \gamma_2}, \frac{\eta}{3}) \right).\]
We know that $B_{\mathbb{P}}(E^s_{\gamma_1 \gamma_2}, \frac{\eta}{3}) \subset B_{\mathbb{P}}(E^s_{\gamma_1}, \frac{2\eta}{3}) \subset \mathbb{P}(E) \setminus X_2^-$, hence
\[\textstyle
\mathcal{L} \left( \gamma_1 \gamma_2,\; B_{\mathbb{P}}(E^s_{\gamma_1 \gamma_2}, \frac{\eta}{3}) \right)
   \leq \mathcal{L} \left( \gamma_1 \gamma_2,\; \mathbb{P}(E) \setminus X_2^- \right).\]
On the other hand, from \eqref{eq:lip_less_s}, it follows that
\[\mathcal{L} \left( \gamma_1 \gamma_2,\; \mathbb{P}(E) \setminus X_2^- \right)
  \lesssim_{C} \tilde{s}(\gamma_1)\tilde{s}(\gamma_2).\]
Stringing together these inequalities, we get
\[\tilde{s}(\gamma_1 \gamma_2) \lesssim_{C} \tilde{s}(\gamma_1)\tilde{s}(\gamma_2);\]
thus (ii) is also proved.
\end{proof}

\subsection{$\mathbb{R}$-regular case}
\label{sec:proximal_to_regular}

The following proposition estimates the position of dynamical spaces and the contraction strength for a product of two sufficiently contracting $\mathbb{R}$-regular maps forming a non-degenerate pair.

\begin{proposition}
\label{regular_product}
For every $C \geq 1$, there is a positive constant $s_1(C) \leq 1$ with the following property. Take any $C$-non-degenerate pair $(g, h)$ of $\mathbb{R}$-regular maps in $G \ltimes \mathfrak{g}$; suppose that the maps $g^{\pm 1}$ and $h^{\pm 1}$ are all $s_1(C)$-contracting. Then $gh$ is $\mathbb{R}$-regular, $2C$-non-degenerate, and we have:
\begin{enumerate}[(i)]
\addtolength{\itemsep}{.5\baselineskip}
\item $\begin{cases}
\alpha^\mathrm{Haus} \left(A^{\subge}_{gh},\; A^{\subge}_{g} \right) \lesssim_C s(g) \vspace{1mm} \\
\alpha^\mathrm{Haus} \left(A^{\suble}_{gh},\; A^{\suble}_{h} \right) \lesssim_C s(h^{-1})
\end{cases}$;
\item $s(gh) \lesssim_C s(g)s(h)$.
\end{enumerate}
\end{proposition}
Recall that the distinction between $s(h^{-1})$ and $s(h)$ is not essential here: see Remark~\ref{contraction_strength_and_inverse}.

Before giving the proof, let us first formulate a particular case:

\begin{corollary}
\label{vector_spaces_estimate}
Under the same hypotheses, we have
\[\begin{cases}
\alpha^\mathrm{Haus} \left(V^{\subge}_{gh},\; V^{\subge}_{g} \right) \lesssim_C s(\ell(g)) \vspace{1mm} \\
\alpha^\mathrm{Haus} \left(V^{\suble}_{gh},\; V^{\suble}_{h} \right) \lesssim_C s(\ell(h)^{-1}).
\end{cases}\]
\end{corollary}
\begin{proof}
If a pair $(g, h)$ satisfies the hypotheses of Proposition \ref{regular_product}, then Lemma \ref{affine_to_vector} shows that the pair $(\ell(g), \ell(h))$ still does. But for every $\mathbb{R}$\nobreakdash-\hspace{0pt}regular $f$, since $\ell(f)$ and $f$ have the same action on $\mathfrak{g}$, obviously we have $V^{\subge}_{\ell(f)} = V^{\subge}_{f}$ and $V^{\suble}_{\ell(f)} = V^{\suble}_{f}$.
\end{proof}

To prove Proposition \ref{regular_product}, we use the result of the previous subsection, by establishing a correspondence between $\mathbb{R}$-regularity and proximality in a suitable exterior power.

We introduce the integers:
\begin{align*}
p &:= \dim \hat{\mathfrak{p}}^+ = \dim \mathfrak{p}^+ + 1; \\
q &:= \dim \mathfrak{n}^-; \\
d &:= \dim \hat{\mathfrak{g}} = \dim \mathfrak{g} + 1 = q+p.
\end{align*}
For every $g \in G \ltimes \mathfrak{g}$, we may define its exterior power $\ext^p g: \ext^p \hat{\mathfrak{g}} \to \ext^p \hat{\mathfrak{g}}$. The Euclidean structure of $\hat{\mathfrak{g}}$ induces in a canonical way a Euclidean structure on $\ext^p \hat{\mathfrak{g}}$.

\begin{lemma} \mbox{ }
\label{regular_to_proximal}
\begin{enumerate}[(i)]
\item For $g \in G \ltimes \mathfrak{g}$, $\ext^{p} g$ is proximal iff $g$ is $\mathbb{R}$-regular. Moreover, the attracting (resp. repelling) space of $\ext^{p} g$ depends on nothing but $A^{\subge}_{g}$ (resp. $V^{\subl}_{g}$):
\begin{equation}
\label{eq:frame_transformation}
\begin{cases}
E^s_{\ext^{p} g} = \ext^{p} A^{\subge}_{g} \\
E^u_{\ext^{p} g} = \setsuch{x \in \ext^{p} \hat{\mathfrak{g}}}
                                        {x \wedge \ext^{q} V^{\subl}_{g} = 0}.
\end{cases}
\end{equation}
\item For every $C \geq 1$, whenever $(g_1, g_2)$ is a $C$-non-degenerate pair of $\mathbb{R}$\nobreakdash-\hspace{0pt}regular maps, $(\ext^p g_1, \ext^p g_2)$ is a $C^p$-non-degenerate pair of proximal maps.
\item For every $C \geq 1$, for every $C$-non-degenerate $\mathbb{R}$-regular map $g \in G \ltimes \mathfrak{g}$, we have
\[s(g) \lesssim_C \tilde{s}(\ext^{p} g).\]
If in addition $s(g) \leq 1$, we have
\[s(g) \asymp_C \tilde{s}(\ext^{p} g).\]
(Recall the Definitions \ref{s_definition} and \ref{s_tilde_definition} of the "contraction strengths" $s(g)$ and $\tilde{s}(\gamma)$, respectively.)
\item For any two $p$-dimensional subspaces $A_1$ and $A_2$ of $\hat{\mathfrak{g}}$, we have
\[\alpha^\mathrm{Haus}(A_1, A_2)
\;\asymp\; \alpha \left( \ext^{p} A_1,\; \ext^{p} A_2 \right).\]
\end{enumerate}
\end{lemma}
\begin{proof} \mbox{ }
\begin{enumerate}[(i)]
\item Let $g \in G \ltimes \mathfrak{g}$. Let $\lambda_{1}, \ldots, \lambda_{d}$ be the eigenvalues of $g$ (acting on $\hat{\mathfrak{g}}$) counted with multiplicity and ordered by nondecreasing modulus. Then we know that the eigenvalues of $\ext^{p} g$ counted with multiplicity are exactly the products of the form $\lambda_{i_1}\cdots\lambda_{i_{p}}$, where $1 \leq i_1 < \ldots < i_{p} \leq d$. As the two largest of them (by modulus) are $\lambda_{q+1} \ldots \lambda_{d}$ and $\lambda_{q}\lambda_{q+2} \ldots \lambda_{d}$, it follows that $\ext^{p} g$ is proximal iff $|\lambda_{q}| < |\lambda_{q+1}|$.

On the other hand, by Claim \ref{Asubge_is_ampa} (i), we know that $\dim A^\sube_g \geq \dim \hat{\mathfrak{l}} = d-2p$, with equality iff $g$ is $\mathbb{R}$-regular. Now since the linear part of $g$ preserves the Killing form, the space $V^\subl_g$ is Killing-orthogonal to $V^\suble_g$, which is supplementary (in $\mathfrak{g}$) to $V^\subg_g$; hence $\dim V^\subl_g \leq \dim V^\subg_g$. By symmetry, we get $\dim V^\subl_g = \dim V^\subg_g$. It follows that $\dim V^\subl_g = \frac{1}{2}(d - \dim A^\sube_g) \leq q$, with equality iff $g$ is $\mathbb{R}$-regular.

In particular, we always have $|\lambda_{q+1}| = 1$. Putting everything together, we conclude that
\[\ext^{p} g \text{ is proximal} \;\iff\; |\lambda_{q}| < 1 \;\iff\; \dim V^\subl_g = q \;\iff\; g \text{ is $\mathbb{R}$-regular.}\]

As for the expression of $E^s$ and $E^u$, it follows immediately by considering a basis that trigonalizes $g$.

\item Take any pair of indices $(i, j) \in \{1, 2\}^2$. Let $\phi$ be an optimal canonizing map for the pair $(A^\subge_{g_i}, A^\suble_{g_j})$. Then we have $\phi(A^\subge_{g_i}) = \hat{\mathfrak{p}}^+$ and (by Claim \ref{Asubge_is_ampa} (iii)) $\phi(V^\subl_{g_j}) = \mathfrak{n}^-$. In the Euclidean structure we have chosen, $\hat{\mathfrak{p}}^+$ is orthogonal to $\mathfrak{n}^-$; hence $\ext^p \hat{\mathfrak{p}}^+$ is orthogonal to the hyperplane $\setsuch{x \in \ext^{p} \hat{\mathfrak{g}}}{x \wedge \ext^{q} \mathfrak{n}^- = 0}$. By the previous point, it follows that $\ext^p \phi$ is a canonizing map for the pair $(E^s_{\ext^p g_i}, E^u_{\ext^p g_j})$. As $\|\ext^p \phi\| \leq \|\phi\|^p$ and similarly for $\phi^{-1}$, the conclusion follows.

\item Let $C \geq 1$, and let $g \in G \ltimes \mathfrak{g}$ be a $C$-non-degenerate $\mathbb{R}$-regular map. First remark the following thing: let $\phi$ be an optimal canonizing map for $g$, and let $g' = \phi g \phi^{-1}$. Then it is clear that $s(g') \asymp_C s(g)$ and $\tilde{s}(\ext^p g') \asymp_C \tilde{s}(\ext^p g)$. Thus we may suppose that $V^{\subl}_{g}$, $A^{\sube}_{g}$ and $V^{\subg}_{g}$ are pairwise orthogonal.

We call \emph{singular values} of $g$ the square roots of the eigenvalues of the map $g^*g$ (where $g^*$ is the adjoint map, with respect to the Euclidean norm). Let $s_1 \leq \cdots \leq s_{p}$ (resp. $s'_1 \leq \cdots \leq s'_q$) be the singular values of $g$ restricted to $A^{\subge}_{g}$ (resp. $V^{\subl}_{g}$), so that $\left\| \restr{g^{-1}}{A^{\subge}_{g}} \right\| = s_1^{-1}$ and $\left\| \restr{g}{V^{\subl}_{g}} \right\| = s'_q$. Since the spaces $A^{\subge}_{g}$ and $V^{\subl}_{g}$ are stable by~$g$ and orthogonal, we get that the singular values of $g$ on the whole space $\hat{\mathfrak{g}}$ are
\[s'_1, \ldots, s'_q, s_1, \ldots, s_{p}\]
(note however that if we do not suppose $s(g) \leq 1$, this list might fail to be sorted in nondecreasing order.) On the other hand, we know that the singular values of $\ext^{p} g$ are products of $p$ distinct singular values of $g$. Since $E^s_{\ext^{p} g}$ is orthogonal to $E^u_{\ext^{p} g}$, we may once again analyze the singular values separately for each subspace. We know that the singular value corresponding to $E^s$ is equal to $s_1 \cdots s_{p}$; we deduce that $\left\| \restr{\ext^{p} g}{E^u} \right\|$ is equal to the maximum of the remaining singular values. In particular it is larger than or equal to $s'_q \cdot s_2 \cdots s_{p}$. On the other hand, if $\lambda$ is the largest eigenvalue of $\ext^{p} g$, then we have
\[|\lambda| = \left| \lambda_{q+1} \cdots \lambda_{d} \right|
            = \left| \det (\restr{g}{A^{\subge}_{g}}) \right|
            = s_1 \cdots s_p\]
(where $\lambda_1, \ldots, \lambda_{d}$ are the eigenvalues of $g$ sorted by nondecreasing modulus). It follows that:
\begin{equation}
\label{eq:s'_lower_bound}
\tilde{s}(\ext^{p} g)
   = \frac{\left\| \restr{\ext^{p} g}{E^u_{\ext^{p} g}} \right\|}
          {|\lambda|}
   \geq \frac{s'_q \cdot s_2 \cdots s_{p}}{s_1 \cdots s_{p}}
   = s'_q s_1^{-1}
   = s(g),
\end{equation}
which is the first estimate we were looking for.

Now suppose that $s(g) \leq 1$. Then we have $s'_q \leq s_1$, which means that the singular values of $\ext^{p} g$ are indeed sorted in the "correct" order. Hence $s'_q \cdot s_2 \cdots s_{p}$ is actually the largest singular value of $\restr{\ext^{p} g}{E^u}$, and the inequality becomes an equality: $\tilde{s}(\ext^{p} g) = s(g)$. The second estimate follows.

\item Let $A_1$ and $A_2$ be two $p$-dimensional subspaces of $\hat{\mathfrak{g}}$. Define
\begin{gather*}
\alpha_1 := \alpha^\mathrm{Haus}(A_1, A_2); \\
\alpha_2 := \alpha(\ext^p A_1, \ext^p A_2).
\end{gather*}
We may find an orthonormal basis $(e_1, \ldots, e_{d})$ of $\hat{\mathfrak{g}}$ such that the subspace~$A_1$ has basis $(e_1, \ldots, e_{p})$ and the subspace~$A_2$ has basis
\[ \left( (\cos \theta_i) e_i + (\sin \theta_i) e_{p+i} \right)_{1 \;\leq\; i \;\leq\; p},\]
for some angles $\frac{\pi}{2} \geq \theta_1 \geq \cdots \geq \theta_q \geq \theta_{q+1} = \cdots = \theta_{p} = 0$ (of course $e_j$ is not defined when $j > d$, but in this formula all such vectors have coefficient~0). In this case, we have $\alpha_1 = \theta_1$ and $\cos \alpha_2 = \prod_{i=1}^{p} \cos \theta_i$, hence
\[(\cos \alpha_1)^{p} \leq \cos \alpha_2 \leq \cos \alpha_1.\]
On the other hand, for every $\theta \in [0, \frac{\pi}{2}]$, we have $\arccos((\cos \theta)^{p}) \leq \sqrt{p}\;\theta$. Indeed, for $\theta \geq \frac{\pi}{2\sqrt{p}}$ this is obvious, and for $\theta \in [0, \frac{\pi}{2\sqrt{p}}]$ this is equivalent to the inequality $p \log \cos \theta \geq \log \cos (\sqrt{p} \theta)$. The latter is clearly true for $\theta = 0$, and follows for the other values of $\theta$ by integrating the inequality $-p \tan \theta \geq -\sqrt{p} \tan (\sqrt{p} \theta)$, which is true by convexity of the tangent function. Finally we get
\begin{samepage}
\[\alpha_1 \leq \alpha_2 \leq \sqrt{p}\;\alpha_1. \qedhere\]
\end{samepage}
\end{enumerate}
\end{proof}

We also need the following technical lemma:

\begin{lemma}
\label{continuity_of_non_degeneracy}
There is a constant $\eps > 0$ with the following property. Let $\hat{\mathfrak{p}}_1, \hat{\mathfrak{p}}_2$ be any two affine m.p.a.'s such that
\[\begin{cases}
\alpha^\mathrm{Haus}(\hat{\mathfrak{p}}_1, \hat{\mathfrak{p}}^+) \leq \eps \\
\alpha^\mathrm{Haus}(\hat{\mathfrak{p}}_2, \hat{\mathfrak{p}}^-) \leq \eps.
\end{cases}\]
Then they form a $2$-non-degenerate pair.
\end{lemma}
(Of course the constant~2 is arbitrary; we could replace it by any number larger than~1.)
\begin{proof}
Let $\mathscr{P}$ be the set of all pairs of affine m.p.a.'s, $\mathscr{P}' \subset \mathscr{P}$ the subset of transverse pairs. Since $\mathscr{P}'$ is an open subset of $\mathscr{P}$, for $\eps$ sufficiently small $\hat{\mathfrak{p}}_1$ and $\hat{\mathfrak{p}}_2$ will be transverse. Moreover, $\mathscr{P}'$ is a homogeneous space under the action of $G \ltimes \mathfrak{g}$ (by Claim \ref{pair_transitivity}), hence the orbital map that maps an element $\phi \in G \ltimes \mathfrak{g}$ to the pair $\phi(\hat{\mathfrak{p}}^+, \hat{\mathfrak{p}}^-)$ is open. It follows that for any~$C$, the set of "strictly $C$-non-degenerate" (meaning $C'$-non-degenerate for some $C' < C$) pairs is open.
\end{proof}

\begin{proof}[Proof of Proposition \ref{regular_product}]
Let $C \geq 1$, and let $(g, h)$ be a $C$-non-degenerate pair of $\mathbb{R}$-regular maps in $G \ltimes \mathfrak{g}$. Suppose that $g^{\pm 1}$ and $h^{\pm 1}$ are $s_1(C)$-contracting, for some constant $s_1(C)$ to be specified later.

Take $\gamma_1 = \ext^{p} g$ and $\gamma_2 = \ext^{p} h$. Let us check the conditions of Proposition \ref{proximal_product}. Indeed:
\begin{itemize}
\item By Lemma \ref{regular_to_proximal} (i), $\gamma_1$ and $\gamma_2$ are proximal.
\item By Lemma \ref{regular_to_proximal} (ii), the pair $(\gamma_1, \gamma_2)$ is $C^p$-non-degenerate.
\item Since we have supposed $s_1(C) \leq 1$, it follows by Lemma \ref{regular_to_proximal} (iii) that $\tilde{s}(\gamma_1) \lesssim_C s(g)$ and $\tilde{s}(\gamma_2) \lesssim_C s(h)$. If we choose $s_1(C)$ sufficiently small, then $\gamma_1$ and $\gamma_2$ are sufficiently contracting to apply Proposition \ref{proximal_product}, namely $\tilde{s}_1(C^p)$-contracting.
\end{itemize}
Now we apply Proposition \ref{proximal_product} to the map $\ext^{p} (gh) = \gamma_1 \gamma_2$.  It remains to deduce the conclusions of Proposition \ref{regular_product}.
\begin{itemize}
\item That $gh$ is $\mathbb{R}$-regular follows by Lemma \ref{regular_to_proximal} (i).
\item From Proposition \ref{proximal_product} (i), using Lemma \ref{regular_to_proximal} (i), (iii) and (iv), we get
\[\alpha^\mathrm{Haus} \left(A^{\subge}_{gh},\; A^{\subge}_{g} \right)
  \lesssim_C s(g),\]
which shows the first line of Proposition \ref{regular_product} (i).
\item By applying Proposition \ref{proximal_product} to $\gamma_2^{-1} \gamma_1^{-1}$ instead of $\gamma_1 \gamma_2$, we get in the same way the second line of Proposition \ref{regular_product} (i).
\item Let $\phi$ be an optimal canonizing map for the pair $(A^{\subge}_{g}, A^{\suble}_{h})$. By hypothesis, we have~$\left \|\phi^{\pm 1} \right\| \leq C$. But if we take $s_1(C)$ sufficiently small, the two inequalities that we have just shown, together with Lemma \ref{continuity_of_non_degeneracy}, allow us to find a map $\phi'$ with $\|\phi'\| \leq 2$, $\|{\phi'}^{-1}\| \leq 2$ and
\[\phi' \circ \phi (A^{\subge}_{gh}, A^{\suble}_{gh}) = (\hat{\mathfrak{p}}^+, \hat{\mathfrak{p}}^-).\]
It follows that the composition map $gh$ is $2C$-non-degenerate.
\item The last inequality, namely Proposition \ref{regular_product}~(ii), now follows from Pro\-po\-si\-tion~\ref{proximal_product}~(ii) by using Lemma \ref{regular_to_proximal}~(iii). \qedhere
\end{itemize}
\end{proof}

\section{Additivity of Margulis invariant}
\label{sec:additivity}

Proposition \ref{invariant_additivity} below is the key ingredient of the paper. It explains how the Margulis invariant behaves under group operations (inverse and composition). The first point is trivial to prove, but still important. The proof of the second point occupies the entirety of this section. We prove it by reducing it successively to Lemma \ref{gghg=}, then to Lemma~\ref{close_to_identity}.

We call $w_0$ any map in $G$ such that $w_0(\mathfrak{p}^+, \mathfrak{p}^-) = (\mathfrak{p}^-, \mathfrak{p}^+)$. (By Claim \ref{u_l_invariant_d_fixed}, the result stated below does not depend on the choice of $w_0$.)

\begin{proposition} \mbox{ }
\label{invariant_additivity}
\begin{enumerate}[(i)]
\item For every $\mathbb{R}$-regular map $g \in G \ltimes \mathfrak{g}$, we have
\[M(g^{-1}) = -w_0(M(g)).\]
\item For every $C \geq 1$, there are positive constants $s_2(C) \leq 1$ and $\mu(C)$ with the following property. Let $g, h \in G \ltimes \mathfrak{g}$ be a $C$-non-degenerate pair of $\mathbb{R}$-regular maps, with $g^{\pm 1}$ and $h^{\pm 1}$ all $s_2(C)$-contracting. Then $gh$ is $\mathbb{R}$-regular, and we have:
\[\|M(gh) - M(g) - M(h)\| \leq \mu(C).\]
\end{enumerate}
\end{proposition}

Let $C \geq 1$. We choose some constant $s_2(C) \leq 1$, small enough to satisfy all the constraints that will appear in the course of the proof. For the remainder of this section, we fix $g, h \in G \ltimes \mathfrak{g}$ a $C$-non-degenerate pair of $\mathbb{R}$-regular maps such that $g^{\pm 1}$ and $h^{\pm 1}$ are $s_2(C)$-contracting.

The following remark will be used throughout this section.
\begin{remark}
\label{ggh_gh_2C}
We may suppose that the pairs $(A^{\subge}_{gh}, A^{\suble}_{gh})$, $(A^{\subge}_{hg}, A^{\suble}_{hg})$, $(A^{\subge}_{g}, A^{\suble}_{gh})$ and $(A^{\subge}_{hg}, A^{\suble}_{g})$ are all $2C$-non-degenerate.
Indeed, recall that (by~Proposition \ref{regular_product}), we have
\[\begin{cases}
  \alpha^\mathrm{Haus} \left(A^{\subge}_{gh},\; A^{\subge}_{g} \right) \lesssim_C s(g) \vspace{1mm} \\
  \alpha^\mathrm{Haus} \left(A^{\suble}_{gh},\; A^{\suble}_{h} \right) \lesssim_C s(h^{-1})
\end{cases}\]
and similar inequalities with $g$ and $h$ interchanged. On the other hand, by hypothesis, $(A^{\subge}_{g}, A^{\suble}_{h})$ is $C$-non-degenerate. If we choose $s_2(C)$ sufficiently small, these four statements then follow from Lemma \ref{continuity_of_non_degeneracy}.
\end{remark}

\begin{proof}[Proof of Proposition \ref{invariant_additivity}] \mbox{ }
\begin{enumerate}[(i)]
\item Considering that $V^{\subge}_{g^{-1}} = V^{\suble}_{g}$ and vice-versa, and that $w_0$ commutes with $\pi_{\mathfrak{z}}$, this is obvious from the definition of the Margulis invariant.

\item If we take $s_2(C) \leq s_1(C)$, then Proposition \ref{regular_product} ensures that $gh$ is $\mathbb{R}$-regular.

\noindent To estimate $M(gh)$, we decompose $gh: A^{\sube}_{gh} \to A^{\sube}_{gh}$ into a product of several maps.

\begin{itemize}
\item We begin by decomposing the product $gh$ into its factors. We have the commutative diagram
\[
\begin{tikzcd}
  A^{\sube}_{gh}
& 
& 
& A^{\sube}_{hg}
    \arrow{lll}{g}
& 
& 
& A^{\sube}_{gh}
    \arrow{lll}{h}
    \arrow[bend right]{llllll}[swap]{gh}
\end{tikzcd}
\]
Indeed, since $hg$ is the conjugate of $gh$ by $h$ and vice-versa, we have $h(A^{\sube}_{gh}) = A^{\sube}_{hg}$ and $g(A^{\sube}_{hg}) = A^{\sube}_{gh}$.
\item Next we factor the map $g: A^{\sube}_{hg} \to A^{\sube}_{gh}$ through the map $g: A^{\sube}_{g} \to A^{\sube}_{g}$, which is better known to us. We have the commutative diagram
\[
\begin{tikzcd}
  A^{\sube}_{gh}
    \arrow{rdd}[swap]{\pi_g}
&
&
& A^{\sube}_{hg}
    \arrow{lll}{g}
    \arrow{ldd}{\pi_g}
\\
\\
& A^{\sube}_{g}
& A^{\sube}_{g}
    \arrow{l}{g}
&
\end{tikzcd}
\]
where $\pi_g$ is the projection onto $A^{\sube}_{g}$ parallel to $V^{\subg}_{g} \oplus V^{\subl}_{g}$. (It commutes with $g$ because $A^{\sube}_{g}$, $V^{\subg}_{g}$ and $V^{\subl}_{g}$ are all invariant by~$g$.)
\item Finally, we decompose again every diagonal arrow from the last diagram into two factors. For any two $\mathbb{R}$-regular maps $u$~and~$v$, we introduce the notation
\[A^{\sube}_{u, v} := A^{\subge}_{u} \cap A^{\suble}_{v}.\]
We call $P_1$ (resp. $P_2$) the projection onto $A^{\sube}_{g, gh}$ (resp. $A^{\sube}_{hg, g}$), still parallel to~$V^{\subg}_{g} \oplus V^{\subl}_{g}$. To justify this definition, we must check that $A^{\sube}_{g, gh}$ (and similarly~$A^{\sube}_{hg, g}$) is supplementary to $V^{\subg}_{g} \oplus V^{\subl}_{g}$. Indeed, by Remark \ref{ggh_gh_2C}, $A^{\suble}_{gh}$~is transverse to $A^{\subge}_{g}$, hence (by Claim \ref{Asubge_is_ampa} (iii)) supplementary to $V^{\subg}_{g}$;
thus $A^{\subge}_{g} = V^{\subg}_{g} \oplus A^{\sube}_{g, gh}$ and $\hat{\mathfrak{g}} = V^{\subl}_{g} \oplus A^{\subge}_{g} = V^{\subl}_{g} \oplus V^{\subg}_{g} \oplus A^{\sube}_{g, gh}$. Then we have the commutative diagrams
\[
\begin{tikzcd}
  A^{\sube}_{gh}
    \arrow[bend right]{rr}[swap]{\pi_g}
    \arrow{r}{P_1}
& A^{\sube}_{g, gh}
    \arrow{r}{\pi_g}
& A^{\sube}_{g}
\end{tikzcd}
\]
and
\[
\begin{tikzcd}
  A^{\sube}_{hg}
    \arrow[bend right]{rr}[swap]{\pi_g}
    \arrow{r}{P_2}
& A^{\sube}_{hg, g}
    \arrow{r}{\pi_g}
& A^{\sube}_{g}
\end{tikzcd}
\]
\end{itemize}

The second and third step can be repeated with $h$ instead of $g$. The way to adapt the second step is straightforward; for the third step, we factor $\pi_h: A^{\sube}_{hg} \to A^{\sube}_{h}$ through $A^{\sube}_{h, hg}$ and $\pi_h: A^{\sube}_{gh} \to A^{\sube}_{h}$ through $A^{\sube}_{gh, h}$.

Combining these three decompositions, we get the lower half of Diagram \ref{fig:largediagcomm1}. (We left out the expansion of $h$; we leave drawing the full diagram for especially brave readers.) Let us now interpret all these maps as endomorphisms of $\hat{\mathfrak{l}}$. To do this, we choose some optimal canonizing maps
\[\phi_{g},\; \phi_{gh},\; \phi_{hg},\; \phi_{g, gh},\; \phi_{hg, g}\]
respectively of $g$, of $gh$, of $hg$, of the pair $(A^{\subge}_{g}, A^{\suble}_{gh})$ and of the pair $(A^{\subge}_{hg}, A^{\suble}_{g})$. This allows us to define $\overline{g_{gh}}$, $\overline{h_{gh}}$, $\overline{g_{g, gh}}$, $\overline{g_{\sube}}$, $\overline{P_1}$, $\overline{P_2}$, $\overline{\psi_1}$,~$\overline{\psi_2}$ to be the maps that make the whole Diagram \ref{fig:largediagcomm1} commutative.

\begin{figure}
\[
\begin{tikzcd}
  \hat{\mathfrak{l}}
    \arrow{rdd}[swap]{\overline{P_1}}
&
&
&
&
&
& \hat{\mathfrak{l}}
    \arrow{ldd}{\overline{P_2}}
    \arrow{llllll}{\overline{g_{gh}}}
& \hat{\mathfrak{l}}
    \arrow{l}{\overline{h_{gh}}}
\\
\\
&
  \hat{\mathfrak{l}}
    \arrow{rdd}[swap]{\overline{\psi_1}}
&
&
&
& \hat{\mathfrak{l}}
    \arrow{ldd}{\overline{\psi_2}}
    \arrow{llll}{\overline{g_{g, gh}}}
&
&
\\
\\
&
& \hat{\mathfrak{l}}
&
& \hat{\mathfrak{l}}
	\arrow{ll}{\overline{g_{\sube}}}
&
&
&
\\
\\
  A^{\sube}_{gh}
    \arrow{rdd}{P_1}
    \arrow[dashed]{uuuuuu}[pos=0.13]{\phi_{gh}}
&
&
&
&
&
& A^{\sube}_{hg}
    \arrow{ldd}[swap]{P_2}
    \arrow{llllll}{g}
    \arrow[dashed]{uuuuuu}[pos=0.13]{\phi_{hg}}
& A^{\sube}_{gh}
    \arrow{l}{h}
    \arrow[dashed]{uuuuuu}[pos=0.13]{\phi_{gh}}
\\
\\
&
  A^{\sube}_{g, gh}
    \arrow{rdd}{\pi_{g}}
    \arrow[dashed]{uuuuuu}{\phi_{g, gh}}
&
&
&
& A^{\sube}_{hg, g}
    \arrow{ldd}[swap]{\pi_{g}}
    \arrow[dashed]{uuuuuu}{\phi_{hg, g}}
&
&
\\
\\
&
& A^{\sube}_{g}
    \arrow[dashed]{uuuuuu}[pos=0.86]{\phi_{g}}
&
& A^{\sube}_{g}
	\arrow{ll}{g}
    \arrow[dashed]{uuuuuu}[pos=0.86]{\phi_{g}}
&
&
&
\end{tikzcd}
\]
\caption{}
\label{fig:largediagcomm1}
\end{figure}

Now let us define
\[\begin{cases}
M_{gh}(g) := \pi_{\mathfrak{z}}(\overline{g_{gh}}(x) - x) \\
M_{gh}(h) := \pi_{\mathfrak{z}}(\overline{h_{gh}}(x) - x)
\end{cases}\]
for any $x \in \mathfrak{l}_{\Aff}$, where $\mathfrak{l}_{\Aff} := \hat{\mathfrak{l}} \cap \mathfrak{g}_{\Aff}$ is the affine space parallel to $\mathfrak{l}$ and passing through the origin. Since $gh$ is the conjugate of $hg$ by $g$ and vice-versa, the maps $\overline{g_{gh}}$ and $\overline{h_{gh}}$ stabilize the spaces $\hat{\mathfrak{p}}^+$ and $\hat{\mathfrak{p}}^-$; by Claim \ref{u_l_invariant_d_fixed}, they are thus quasi-translations. It follows that these values $M_{gh}(g)$ and $M_{gh}(h)$ do not depend on the choice of $x$. Compare this to the alternative formula \eqref{eq:margulis_invariant_alternative} for the Margulis invariant: we have $M(gh) = \pi_{\mathfrak{z}}(\overline{g_{gh}} \circ \overline{h_{gh}}(x) - x)$ for any $x \in \mathfrak{l}_{\Aff}$. It immediately follows that
\[M(gh) = M_{gh}(g) + M_{gh}(h).\]
Thus it is enough to show the estimates $\|M_{gh}(g) - M(g)\| \lesssim_C 1$ and $\|M_{gh}(h) - M(h)\| \lesssim_C 1$. This is an immediate consequence of Lemma \ref{gghg=} below. (Note that while the vectors $M_{gh}(g)$ and $M_{gh}(h)$ are elements of $\mathfrak{z}$, the maps $\overline{g_{gh}}$ and $\overline{h_{gh}}$ are extended affine isometries acting on the whole subspace $\hat{\mathfrak{l}}$.) \qedhere
\end{enumerate}
\end{proof}
\begin{remark}
In contrast to actual Margulis invariants, the values $M_{gh}(g)$ and $M_{gh}(h)$ \emph{do} depend on our choice of canonizing maps. Choosing other canonizing maps would force us to subtract some constant from the former and add it to the latter.
\end{remark}

\begin{definition}
We shall say that a linear bijection $f$ between two subspaces of $\hat{\mathfrak{g}}$
is \emph{$K(C)$-bounded} if it is bounded by a constant depending only on $C$, that is, $\|f\| \lesssim_C 1$ and $\|f^{-1}\| \lesssim_C 1$. We say that two automorphisms $f_1, f_2$ of $\hat{\mathfrak{l}}$ (depending somehow on $g$ and $h$) are \emph{$K(C)$-almost equivalent}, and we write $f_1 \approx_C f_2$, if they satisfy the condition
\[\|f_1 - \xi \circ f_2 \circ \xi'\| \lesssim_C 1\]
for some $K(C)$-bounded quasi-translations $\xi, \xi'$. This is indeed an equivalence relation.
\end{definition}

\begin{lemma}
\label{gghg=}
The maps $\overline{g_{gh}}$ and $\overline{h_{gh}}$ are $K(C)$-almost equivalent to $\overline{g_{\sube}}$ and $\overline{h_{\sube}}$, respectively.
\end{lemma}

To show this, we use the following property:
\begin{lemma}
\label{C-bounded}
All the non-horizontal arrows in Diagram \ref{fig:largediagcomm1} represent $K(C)$-bounded, bijective maps.
\end{lemma}

Note that Lemma \ref{C-bounded} alone does not imply Lemma \ref{gghg=}: indeed, while the maps $\overline{\psi}_1$ and $\overline{\psi}_2$ are quasi-translations by Lemma \ref{projections_commute}, the maps $\overline{P}_1$~and~$\overline{P}_2$ need not be. This issue will be addressed in Lemma \ref{close_to_identity}. 

\begin{proof}[Proof of Lemma \ref{C-bounded}]\mbox{ }
\begin{itemize}
\item For the vertical arrows, this is an immediate consequence of Remark \ref{ggh_gh_2C}.
\item Let us take care of the maps $P_1: A^\sube_{gh} \to A^\sube_{g, gh}$, $P_2: A^\sube_{hg} \to A^\sube_{hg, g}$, $\pi_g: A^\sube_{g, gh} \to A^\sube_{g}$ and $\pi_g: A^\sube_{hg, g} \to A^\sube_{g}$. All of these maps are projections parallel to $V^\subg_g \oplus V^\subl_g$; thus to show that they are bijective and $K(C)$-bounded, it is enough to give a positive lower bound, depending only on $C$, on the five angles between $V^\subg_g \oplus V^\subl_g$ and each of the five subspaces
\[A^\sube_{gh},\; A^\sube_{g, gh},\; A^\sube_{g},\; A^\sube_{hg, g},\; A^\sube_{hg}.\]
(For bijectivity alone it would be enough to check that these angles are positive, \ie that each of these five subspaces is supplementary to $V^\subg_g \oplus V^\subl_g$. This is obvious for $A^\sube_{g}$, and has already been done for $A^\sube_{g, gh}$ and $A^\sube_{hg, g}$ to justify that $P_1$ and $P_2$ are well-defined.)

Let us estimate these angles:
\begin{itemize}
\item The fact that
\[\alpha(A^\sube_{g},\; V^\subg_g \oplus V^\subl_g) \gtrsim_C 1\]
is a direct consequence of the fact that $g$ is $C$-non-degenerate.
\item Let us estimate the position of $A^\sube_{g, gh}$. We know that $\phi_{g, gh}$ sends, respectively, $A^\suble_{gh}$ and $V^\subg_g$ to $\hat{\mathfrak{p}}^-$ and $\mathfrak{n}^+$ (using Claim~\ref{Asubge_is_ampa} (iii) about uniqueness of $V^\subg$). By convention, the latter two spaces are orthogonal; since the pair $(A^\subge_g, A^\suble_{gh})$ is $2C$-non-degenerate (Remark \ref{ggh_gh_2C}), it follows by Lemma \ref{bounded_norm_is_bilipschitz} about the Lipschitz constant of bounded maps that
\[\alpha(A^\suble_{gh}, V^\subg_g) \gtrsim_C 1\]
(in fact the left-hand side is precisely bounded below by $\frac{1}{4C^2}\frac{\pi}{2}$), and in particular
\[\alpha(A^\sube_{g, gh}, V^\subg_g) \gtrsim_C 1.\]
Next we apply the map $\phi_g$; since $\| \phi_g \| \leq C$ and $\| \phi_g^{-1} \| \leq C$, the distance is, once again, divided by at most $C^2$. But after applying this map, the space $\phi_g(A^\subge_g) = \hat{\mathfrak{p}}^+$, containing both $\phi_g(A^\sube_{g, gh})$ and $\phi_g(V^\subg_g)$, is orthogonal to $\phi_g(V^\subl_g) = \mathfrak{n}^+$; hence we have
\[\alpha \left( \phi_g(A^\sube_{g, gh}),\; \phi_g(V^\subg_g) \right)
 = \alpha \left( \phi_g(A^\sube_{g, gh}),\; \phi_g(V^\subg_g \oplus V^\subl_g) \right).\]
Applying $\phi_g^{-1}$ to get the original spaces, we introduce again a factor no smaller than $\frac{1}{C^2}$. We conclude that
\[\alpha(A^\sube_{g, gh}, V^\subg_g \oplus V^\subl_g) \gtrsim_C 1\]
as required.
\item For $A^\sube_{hg, g}$, by symmetry, the same calculation holds, \emph{mutatis mutandis}. By applying the map $\phi_{hg, g}$, we find similarly that
\[\alpha(A^\subge_{hg}, V^\subl_g) \gtrsim_C 1,\]
hence in particular
\[\alpha(A^\sube_{hg, g}, V^\subl_g) \gtrsim_C 1.\]
When we apply $\phi_g$, we get that
\[\alpha(\phi_g(A^\sube_{hg, g}),\; \phi_g(V^\subl_g))
= \alpha(\phi_g(A^\sube_{hg, g}),\; \phi_g(V^\subg_g \oplus V^\subl_g));\]
applying $\phi_g^{-1}$, we conclude that
\[\alpha(A^\sube_{hg, g}, V^\subg_g \oplus V^\subl_g) \gtrsim_C 1.\]
\item To show that the space $A^\sube_{gh}$ is "far" from $V^\subg_g \oplus V^\subl_g$, we will show that it is "close" to the space $A^\sube_{g, gh}$, that we have already shown to be "far" from $V^\subg_g \oplus V^\subl_g$.

We shall use the following property: if the linear subspaces $F$ and $G$ are perpendicular (meaning that the orthogonal supplements of $F \cap G$ respectively in $F$ and $G$ are orthogonal to each other), then for any subspace $F'$, we have
\[\alpha^\mathrm{Haus}(F \cap G, F' \cap G) \leq \alpha^\mathrm{Haus}(F, F'),\]
provided that $F' \cap G$ still has the same dimension as $F \cap G$.

Taking as $F$, $G$ and $F'$ the images by the $K(C)$-bounded map $\phi_{gh}$ of the spaces $A^{\subge}_{gh}$, $A^{\suble}_{gh}$ and $A^{\subge}_g$ respectively, we deduce that
\[\alpha^\mathrm{Haus} \left(A^{\sube}_{gh},\; A^{\sube}_{g, gh} \right)
\lesssim_C \alpha^\mathrm{Haus} \left(A^{\subge}_{gh},\; A^{\subge}_g \right).\]

On the other hand, Proposition \ref{regular_product} tells us that
\[\alpha^\mathrm{Haus} \left(A^{\subge}_{gh},\; A^{\subge}_g \right) \lesssim_C s(g).\]

Since $s(g) \leq s_2(C)$, taking $s_2(C)$ small enough, we may suppose that the Hausdorff distance between $A^{\sube}_{gh}$ and $A^{\sube}_{g, gh}$ is less than half our lower bound for the minimal distance between $A^{\sube}_{g, gh}$ and $V^\subg_g \oplus V^\subl_g$. We conclude that
\[\alpha(A^\sube_{gh}, V^\subg_g \oplus V^\subl_g) \gtrsim_C 1.\]
\item To estimate the position of $A^\sube_{hg}$, similarly, we show that it is close to $A^\sube_{hg, g}$. We apply the same property as above, taking now as $F$, $G$ and $F'$ the images by $\phi_{hg}$ of $A^{\suble}_{hg}$, $A^{\subge}_{hg}$ and $A^{\suble}_g$ respectively. We deduce that 
\[\alpha^\mathrm{Haus} \left(A^{\sube}_{hg},\; A^{\sube}_{hg, g} \right)
\lesssim_C \alpha^\mathrm{Haus} \left(A^{\suble}_{hg},\; A^{\suble}_g \right).\]
Using once again Proposition \ref{regular_product} and choosing $s_2(C)$ small enough, we conclude that
\[\alpha(A^\sube_{hg}, V^\subg_g \oplus V^\subl_g) \gtrsim_C 1.\]
\end{itemize}
This shows that all the diagonal arrows in the lower half of the diagram represent $K(C)$-bounded bijections.
\item The maps $\overline{P_1}$, $\overline{P_2}$, $\overline{\psi_1}$ and $\overline{\psi_2}$ from the upper half of the diagram are now compositions of $K(C)$-bounded bijections, hence they are themselves bijective and $K(C)$-bounded. This completes the proof of Lemma \ref{C-bounded}.
\qedhere
\end{itemize}
\end{proof}

\begin{proof}[Proof of Lemma \ref{gghg=}]
We shall concentrate on the estimate $\overline{g_{gh}} \approx_C \overline{g_{\sube}}$; the proof of the estimate $\overline{h_{gh}} \approx_C \overline{h_{\sube}}$ is analogous.

According to Lemma \ref{projections_commute}, the maps $\overline{\psi_1}$ and $\overline{\psi_2}$ are quasi-translations. Hence $\overline{g_{g, gh}}$ is also a quasi-translation.

We would like to pretend that $\overline{g_{gh}}$ and $\overline{g_{g, gh}}$ are actually translations. To do that, we modify slightly the upper right-hand corner of Diagram \ref{fig:largediagcomm1}. We set
\[\begin{cases}
\phi'_{hg} := \ell(\overline{g_{gh}}) \circ \phi_{hg} \\
\phi'_{hg, g} := \ell(\overline{g_{g, gh}}) \circ \phi_{hg, g},
\end{cases}\]
where $\ell$ stands for the linear part as defined in Section \ref{sec:affine_to_linear}, and we define $\overline{P'_2}$, $\overline{\psi'_2}$, $\overline{g'_{gh}}$, $\overline{g'_{g, gh}}$ so as to make the new diagram commutative (see Diagram~\ref{fig:largediagcomm2}). The factors $\ell(\overline{g_{gh}})$ and $\ell(\overline{g_{g, gh}})$ we introduced (the short horizontal arrows in Diagram \ref{fig:largediagcomm2}) have norm 1: indeed, being quasi-translations of $\hat{\mathfrak{l}}$ fixing~$\mathbb{R}_0$, they are orthogonal linear transformations (by Proposition~\ref{quasi-translation}). Thus Lemma \ref{C-bounded} still holds for Diagram \ref{fig:largediagcomm2}; but now, the modified maps $\overline{g'_{gh}}$ and $\overline{g'_{g, gh}}$ are translations by construction.

\begin{figure}[t]
\[
\definecolor{mygray}{gray}{0.7}
\begin{tikzcd}
  \hat{\mathfrak{l}}
    \arrow{rdd}[swap]{\overline{P_1}}
&
&
&
&
& \hat{\mathfrak{l}}
    \arrow{ldd}[swap]{\overline{P'_2}}
    \arrow{lllll}{\overline{g'_{gh}}}
& {\color{mygray}\hat{\mathfrak{l}}}
    \arrow[color=mygray]{ldd}[color=mygray]{\overline{P_2}}
    \arrow[dashed]{l}{\ell(\overline{g_{gh}})}
&
\\
\\
&
  \hat{\mathfrak{l}}
    \arrow{rdd}[swap]{\overline{\psi_1}}
&
&
& \hat{\mathfrak{l}}
    \arrow{dd}[swap]{\overline{\psi'_2}}
    \arrow{lll}{\overline{g'_{g, gh}}}
& {\color{mygray}\hat{\mathfrak{l}}}
    \arrow[color=mygray]{ldd}[color=mygray]{\overline{\psi_2}}
    \arrow[dashed]{l}{\ell(\overline{g_{g, gh}})}
&
&
\\
\\
&
& \hat{\mathfrak{l}}
&
& \hat{\mathfrak{l}}
	\arrow{ll}{\overline{g_{\sube}}}
&
&
&
\\
\\
  A^{\sube}_{gh}
    \arrow{rdd}{P_1}
    \arrow[dashed]{uuuuuu}[pos=0.13]{\phi_{gh}}
&
&
&
&
&
& A^{\sube}_{hg}
    \arrow{ldd}[swap]{P_2}
    \arrow{llllll}{g}
    \arrow[dashed]{luuuuuu}[pos=0.17]{\phi'_{hg}}
    \arrow[dashed, color=mygray]{uuuuuu}[swap, pos=0.13, color=mygray]{\phi_{hg}}
&
\\
\\
&
  A^{\sube}_{g, gh}
    \arrow{rdd}{\pi_{g}}
    \arrow[dashed]{uuuuuu}{\phi_{g, gh}}
&
&
&
& A^{\sube}_{hg, g}
    \arrow{ldd}[swap]{\pi_{g}}
    \arrow[dashed]{luuuuuu}[swap, pos=0.45]{\phi'_{hg, g}}
    \arrow[dashed, color=mygray]{uuuuuu}[swap, color=mygray]{\phi_{hg, g}}
&
&
\\
\\
&
& A^{\sube}_{g}
    \arrow[dashed]{uuuuuu}[pos=0.86]{\phi_{g}}
&
& A^{\sube}_{g}
	\arrow{ll}{g}
    \arrow[dashed]{uuuuuu}[pos=0.86]{\phi_{g}}
&
&
&
\end{tikzcd}
\]
\caption{}
\label{fig:largediagcomm2}
\end{figure}

We may write:
\[\overline{g'_{gh}} = (\overline{P_1}^{-1} \circ \overline{g'_{g, gh}} \circ \overline{P_1}) \circ (\overline{P_1}^{-1} \circ \overline{P'_2}).\]
Then, since $\overline{g'_{gh}}$ and $\overline{g'_{g, gh}}$ are translations, $\overline{P_1}^{-1} \circ \overline{P'_2}$ is also a translation. By Lemma~\ref{C-bounded} (applied to Diagram \ref{fig:largediagcomm2}), it is the composition of two $K(C)$-bounded maps, hence $K(C)$-bounded. Thus we have
\[\overline{g'_{gh}} \approx_C \overline{P_1}^{-1} \circ \overline{g'_{g, gh}} \circ \overline{P_1}.\]
Since $\ell(\overline{g_{gh}})$, $\ell(\overline{g_{g, gh}})$, $\overline{\psi_1}$ and $\overline{\psi_2}$ are $K(C)$-bounded quasi-translations, $\overline{g_{gh}}$~is $K(C)$-almost equivalent to $\overline{g'_{gh}}$ and $\overline{g_{\sube}}$ is $K(C)$-almost equivalent to $\overline{g'_{g, gh}}$. It remains to check that the map $\overline{g'_{g, gh}}$ is $K(C)$-almost equivalent to its conjugate $\overline{P_1}^{-1} \circ \overline{g'_{g, gh}} \circ \overline{P_1}$.

This follows from Lemma \ref{close_to_identity} below. Indeed, let us call $\overline{P''_1}$ the quasi-translation constructed in Lemma \ref{close_to_identity}. Let $v \in \mathfrak{l}$ be the translation vector of $\overline{g'_{g, gh}}$, so that
\[\overline{g'_{g, gh}} = \tau_v.\]
\begin{samepage}
Then we have
\[\left\| \overline{P_1}^{-1} \circ \overline{g'_{g, gh}} \circ \overline{P_1}
 - \overline{P''_1}^{-1} \circ \overline{g'_{g, gh}} \circ \overline{P''_1} \right\|
= \left\| \tau_{\overline{P_1}^{-1}(v)} - \tau_{\overline{P''_1}^{-1}(v)} \right\|.\]
Keep in mind that, for any vector $u$, while we call the map $\tau_u$ a "translation", it is actually a transvection; so its norm $\| \tau_u \|$ is equal to the norm of the matrix $\bigl( \begin{smallmatrix} \Id & \|u\| \\ 0 & 1 \end{smallmatrix} \bigr)$. In particular we have $\|u\| \leq \|\tau_u\| \leq \|u\| + 1$. It follows that
\end{samepage}
\begin{samepage}
\begin{align*}
\left\| \tau_{\overline{P_1}^{-1}(v)} - \tau_{\overline{P''_1}^{-1}(v)} \right\|
&\leq \left\| \overline{P_1}^{-1}(v) - \overline{P''_1}^{-1}(v) \right\| + 1 \\
&\leq \left\| \restr{(\overline{P_1}^{-1} - \overline{P''_1}^{-1})}{\mathfrak{l}} \right\| \| v \| + 1
\end{align*}
(as $v \in \mathfrak{l}$).
\end{samepage}

Now by Proposition \ref{quasi-translation}, we know that the quasi-translation $\overline{P''_1}$ restricted to $\mathfrak{l}$ is an element of the group $D$; since it is compact, the map $\rho \mapsto \rho^{-1}$ is Lipschitz-continuous on that group. Then we may deduce from Lemma \ref{close_to_identity} that
\[\left \| \restr{(\overline{P_1}^{-1} - \overline{P''_1}^{-1})}{\mathfrak{l}} \right \| \lesssim_C s(\ell(g)).\]
On the other hand, we have $\|v\| \leq \|\tau_v\| = \left\| \overline{g'_{g, gh}} \right\| \lesssim_C \left\| \restr{g}{A^\sube_g} \right\|$, since $\overline{g'_{g, gh}}$ is the composition of $\restr{g}{A^\sube_g}$ with several $K(C)$-bounded maps. It follows that
\[\left\| \overline{P_1}^{-1} \circ \overline{g'_{g, gh}} \circ \overline{P_1}
 - \overline{P''_1}^{-1} \circ \overline{g'_{g, gh}} \circ \overline{P''_1} \right\|
\lesssim_C s(\ell(g)) \left\|\restr{g}{A^\sube_g}\right\| + 1.\]
By Lemma \ref{affine_to_vector} (iii), we have $s(\ell(g)) \left\|\restr{g}{A^\sube_g}\right\| \lesssim_C s(g)$; and we know that $s(g) \leq 1$. Finally we get
\[\left\| \overline{P_1}^{-1} \circ \overline{g'_{g, gh}} \circ \overline{P_1}
 - \overline{P''_1}^{-1} \circ \overline{g'_{g, gh}} \circ \overline{P''_1} \right\|
\lesssim_C 1.\]
To complete the proof of Lemma \ref{gghg=}, and hence also the proof of Proposition \ref{invariant_additivity}, it remains only to prove Lemma \ref{close_to_identity}.
\end{proof}

\begin{lemma}
\label{close_to_identity}
The linear part of the map $\overline{P_1}$ is "almost" a quasi-translation. More precisely, there is a quasi-translation $\overline{P''_1}$ such that
\[\left \| \restr{(\overline{P_1} - \overline{P''_1})}{\mathfrak{l}} \right \| \lesssim_C s(\ell(g)).\]
\end{lemma}
Recall that $\ell(g)$ is the map with the same linear part as $g$, but with no translation part: see subsection \ref{sec:affine_to_linear}. We use the double prime because the relationship between $\overline{P''_1}$ and $\overline{P_1}$ is not the same as the relationship between $\overline{P'_2}$ and $\overline{P_2}$.

\begin{proof}
Let $P''_1: A^{\sube}_{gh} \to A^{\sube}_{g, gh}$ be the projection parallel to $V^{\subl}_{gh}$ (recall that $P_1$, by contrast, was along $V^{\subg}_g \oplus V^{\subl}_g$); we set $\overline{P''_1} := \phi_{g, gh} \circ P''_1 \circ \phi_{gh}^{-1}$ the corresponding endomorphism of~$\hat{\mathfrak{l}}$. Then by Lemma \ref{projections_commute}, $\overline{P''_1}$ is a quasi-translation.

We need to show that for any $x \in \mathfrak{l}$, we have
\[\|\overline{P_1}(x) - \overline{P''_1}(x)\| \lesssim_C s(\ell(g))\|x\|.\]
By Remark \ref{ggh_gh_2C}, this is true iff for any $x \in V^{\sube}_{gh}$, we have
\[\|P_1(x) - P''_1(x)\| \lesssim_C s(\ell(g))\|x\|.\]
Take any $x \in V^{\sube}_{gh}$. Let us decompose it in two ways:
\begin{align*}
x &=: \underbrace{x_1}_{\in V^{\sube}_{g, gh}}
   +  \underbrace{x_2}_{\in V^{\subl}_{g}}
   +  \underbrace{x_3}_{\in V^{\subg}_{g}} \\
  &=: \underbrace{x'_1}_{\in V^{\sube}_{g, gh}}
   +  \underbrace{x'_2}_{\in V^{\subl}_{gh}},
\end{align*}
so that $x_1 = P_1(x)$ and $x'_1 = P''_1(x)$. Our first goal is to establish the estimate \eqref{eq:x_2_estimate} below. Roughly, the idea is that since $x \in V^{\sube}_{gh} \subset V^{\subge}_{gh}$, and since the latter subspace is "close" to $V^{\subge}_{g} = V^{\sube}_{g, gh} \oplus V^{\subg}_{g}$, the component $x_2$ is "small".

More precisely, $x_2$ is the image of $x$ by the projection onto $V^{\subl}_{g}$ parallel to~$V^{\subge}_{g}$; hence $\phi_{g}(x_2)$ is the image of $\phi_{g}(x)$ by the projection onto $\mathfrak{n}^-$ parallel to $\mathfrak{p}^+$, which is an orthogonal projection. It follows that
\begin{align*}
\frac{\|\phi_{g}(x_2)\|}{\|\phi_{g}(x)\|}
   &= \sin \alpha \left( \phi_{g}(x), \mathfrak{p}^+ \right) \\
   &\leq \alpha \left( \phi_{g}(x), \mathfrak{p}^+ \right) \\
   &\leq \alpha^\mathrm{Haus} \left(
           \phi_{g}(V^{\subge}_{gh}),
           \phi_{g}(V^{\subge}_{g})
         \right).
\end{align*}
Since $g$ is $C$-non-degenerate, using Lemma \ref{bounded_norm_is_bilipschitz} we get
\[\|x_2\| \lesssim_C \alpha^\mathrm{Haus} (V^{\subge}_{gh}, V^{\subge}_{g}) \|x\|.\]
From Corollary \ref{vector_spaces_estimate}, it follows that
\begin{equation}
\label{eq:x_2_estimate}
\|x_2\| \lesssim_C s(\ell(g)) \|x\|.
\end{equation}

On the other hand, we have:
\[x_2 = \underbrace{(x'_1-x_1)}_{\in V^{\sube}_{g, gh}}
      - \underbrace{x_3}_{\in V^{\subg}_{g}}
      + \underbrace{x'_2}_{\in V^{\subl}_{gh}},\]
hence
\[\|\phi_{g, gh}(x_2)\|^2 =
    \|\phi_{g, gh}(x'_1-x_1)\|^2
  + \|\phi_{g, gh}(x_3)\|^2
  + \|\phi_{g, gh}(x'_2)\|^2\]
and in particular
\[\|\phi_{g, gh}(x'_1-x_1)\| \leq \|\phi_{g, gh}(x_2)\|.\]
Since $(A^{\subge}_{g}, A^{\suble}_{gh})$ is $2C$-non-degenerate (Remark \ref{ggh_gh_2C}), it follows that
\begin{equation}
\label{eq:x'_1-x_1_estimate}
\|x'_1-x_1\| \lesssim_C \|x_2\|.
\end{equation}
\begin{samepage}
Combining this with \eqref{eq:x_2_estimate}, we get
\[\|x'_1-x_1\| \lesssim_C s(\ell(g)) \|x\|;\]
Lemma \ref{close_to_identity}, and therefore Lemma \ref{gghg=} and Proposition \ref{invariant_additivity}, are now proved.
\end{samepage}
\end{proof}

\section{Margulis invariants of words}
\label{sec:induction}

We have already studied how contraction strengths (Proposition \ref{regular_product}) and Margulis invariants (Proposition \ref{invariant_additivity}) behave when we take the product of two $\mathbb{R}$-regular, $C$-non-degenerate, sufficiently contracting maps. The goal of this section is to generalize these results to words of arbitrary length on a given set of generators.

\begin{definition}
\label{cyclically_reduced_definition}
Take $k$ generators $g_1, \ldots, g_k$. Consider a word $g = g_{i_1}^{\sigma_1} \cdots g_{i_l}^{\sigma_l}$ of length $l \geq 1$ on these generators and their inverses (for every~$m$ we have $1 \leq i_m \leq k$ and $\sigma_m = \pm 1$). We say that $g$ is \emph{reduced} if for every~$m$ such that $1 \leq m \leq l-1$, we have $(i_{m+1}, \sigma_{m+1}) \neq (i_m, -\sigma_m)$. We say that $g$ is \emph{cyclically reduced} if it is reduced and also satisfies $(i_1, \sigma_1) \neq (i_l, -\sigma_l)$.
\end{definition}

\begin{proposition}
\label{Schottky_group}
For every $C \geq 1$, there is a positive constant $s_3(C) \leq 1$ with the following property. Take any family of maps $g_1, \ldots, g_k \in G \ltimes \mathfrak{g}$ satisfying the following hypotheses:
\begin{enumerate}[(H1)]
\item Every $g_i$ is $\mathbb{R}$-regular.
\item Any pair taken among the maps $\{g_1, \ldots, g_k, g_1^{-1}, \ldots, g_k^{-1}\}$ is $C$\nobreakdash-\hspace{0pt}non-degenerate, except of course if it has the form $(g_i, g_i^{-1})$ for some~$i$.
\item For every $i$, we have $s(g_i) \leq s_3(C)$ and $s(g_i^{-1}) \leq s_3(C)$.
\end{enumerate}
Take any nonempty cyclically reduced word $g = g_{i_1}^{\sigma_1} \cdots g_{i_l}^{\sigma_l}$ (where we have $1 \leq i_m \leq k$ and $\sigma_m = \pm 1$ for every $m$). Then $g$ is $\mathbb{R}$-regular, $2C$-non-degenerate, and we have
\[\left\| M(g) - \sum_{m=1}^{l} M(g_{i_m}^{\sigma_m}) \right\| \leq l\mu(2C)\]
(where $\mu(2C)$ is the constant introduced in Proposition \ref{invariant_additivity}).
\end{proposition}

The proof proceeds of course by induction, with Proposition \ref{regular_product} and Proposition \ref{invariant_additivity} providing the induction step; however, there is a subtlety. When we suppose that the pair $(g, h)$ is $C$-non-degenerate, we can only conclude that $gh$ is $2C$-non-degenerate; this would break the induction if we used a direct approach. To guarantee $2C$-non-degeneracy for all words, we must use the fact that the contraction strength of $g$ grows exponentially with its length, so that the (Hausdorff) distance between $A^{\subge}_{g}$ and $A^{\subge}_{g_{i_1}^{\sigma_1}}$ is in fact a sum of exponentially diminishing increments and remains bounded. To take this into account, we shall prove by induction a series of slightly more complicated statements.

\begin{proof}
Let us fix $C \geq 1$, a positive constant $s_3(C) \leq 1$ to be determined in the course of the proof, and a family $g_1, \ldots, g_k$ satisfying the hypotheses (H1), (H2) and (H3). We will show by induction on $\max(l, l')$ that whenever we take a nonempty cyclically reduced word $g = g_{i_1}^{\sigma_1} \cdots g_{i_l}^{\sigma_l}$, we have the following properties:
\begin{enumerate}[(i)]
\addtolength{\itemsep}{.5\baselineskip}
\item The map $g$ is $\mathbb{R}$-regular.
\item $\begin{cases}
       \alpha^\mathrm{Haus} \left(A^{\subge}_{g},\;
                                  A^{\subge}_{g_{i_1}^{\sigma_1}} \right)
          \lesssim_C 2 \left( 1 - 2^{-(l-1)} \right) s_3(C) \vspace{1mm} \\
       \alpha^\mathrm{Haus} \left(A^{\suble}_{g},\;
                                  A^{\suble}_{g_{i_l}^{\sigma_l}} \right)
          \lesssim_C 2 \left( 1 - 2^{-(l-1)} \right) s_3(C).
       \end{cases}$
\item $s(g) \leq 2^{-(l-1)} s_3(C)$.
\item $\displaystyle \left\| M(g) - \sum_{m=1}^{l} M(g_{i_m}^{\sigma_m}) \right\| \leq (l-1)\mu(2C)$.
\item If $h = g_{i'_1}^{\sigma'_1} \cdots g_{i'_{l'}}^{\sigma'_{l'}}$ is another nonempty cyclically reduced word such that $gh$ is also cyclically reduced, the pair $(g, h)$ is $2C$-non-degenerate.
\end{enumerate}
In particular, the properties (i), (iv) and (v) imply the Proposition.

Indeed, all five statements are true for $l=1$ (and $l'=1$). Now let $l \geq 2$, and suppose that statements (i) through (v) are true for all cyclically reduced words of length $m$ with $1 \leq m \leq l-1$. Take any cyclically reduced word $g = g_{i_1}^{\sigma_1} \cdots g_{i_l}^{\sigma_l}$. Then we claim that it is possible to decompose $g$ into two cyclically reduced subwords
\[g' := g_{i_1}^{\sigma_1} \cdots g_{i_m}^{\sigma_m}
\quad \text{and} \quad
  g'' := g_{i_{m+1}}^{\sigma_{m+1}} \cdots g_{i_l}^{\sigma_l},\]
both nonempty (that is, $0 < m < l$).

Indeed, suppose the contrary: suppose that for every such $m$, we have
\[(i_m, \sigma_m) = (i_1, -\sigma_1)
\quad \text{ or } \quad
(i_{m+1}, \sigma_{m+1}) = (i_l, -\sigma_l).\]
Let us show, by induction on $m$, that the first condition always fails, hence the second always holds. For $m=1$, this is obvious. Suppose we know it for $m-1$; then we have
\[(i_m, \sigma_m) = (i_l, -\sigma_l) = (i_2, \sigma_2) \neq (i_1, -\sigma_1)\]
because the word is reduced. Now taking $m = l-1$, we get a contradiction.

By induction hypotheses (i) and (v), $g'$ and $g''$ are $\mathbb{R}$-regular and form a $2C$-non-degenerate pair; by induction hypothesis (iii), we have $s(g') \leq 2^{-(m-1)} s_3(C) \leq s_3(C)$ and we may suppose that $s_3(C) \leq s_1(2C)$ (similarly for $g'^{-1}$, $g''$, $g''^{-1}$). Thus the pair $(g', g'')$ satisfies Proposition \ref{regular_product} (with constant $2C$). Let us show that $g$ satisfies the properties (i) through (v).
\begin{itemize}
\item The property (i) (that $g$ is $\mathbb{R}$-regular) is a direct consequence of Proposition \ref{regular_product}.

\item Let us check the property (iii). From Proposition \ref{regular_product} (ii), it follows that $s(g) \lesssim_C s(g')s(g'')$; we then have, by induction hypothesis (iii):
\begin{align*}
s(g) &\lesssim_C \left( 2^{-(m-1)} s_3(C) \right)\left( 2^{-(l-m-1)} s_3(C) \right) \\
     &=          s_3(C) \left( 2^{-(l-2)} s_3(C) \right).\intertext{
If we take $s_3(C)$ sufficiently small, we get
}s(g) &\leq 2^{-(l-1)} s_3(C).\end{align*}

\item Now we check (ii); it is enough to check the first inequality (the second one follows by substituting $g^{-1}$.) Remember that $\alpha^\mathrm{Haus}$ is a metric on the set of all vector subspaces of $\hat{\mathfrak{g}}$, so we have
\begin{align*}
  \alpha^\mathrm{Haus} \left(A^{\subge}_{g},\;
                             A^{\subge}_{g_{i_1}^{\sigma_1}} \right)
&\leq\;
  \alpha^\mathrm{Haus} \left(A^{\subge}_{g},\;
                             A^{\subge}_{g'} \right)
+ \alpha^\mathrm{Haus} \left(A^{\subge}_{g'},\;
                             A^{\subge}_{g_{i_1}^{\sigma_1}} \right).\intertext{
Estimating the first term by Proposition \ref{regular_product} (i) and the second term by induction hypothesis (ii), we get:
} \alpha^\mathrm{Haus} \left(A^{\subge}_{g},\;
                             A^{\subge}_{g_{i_1}^{\sigma_1}} \right)
&\lesssim_C\, s(g') + 2 \left( 1 - 2^{-(m-1)} \right) s_3(C).\intertext{
Now by induction hypothesis (iii) we have $s(g') \leq 2^{-(m-1)} s_3(C)$, hence
} \alpha^\mathrm{Haus} \left(A^{\subge}_{g},\;
                             A^{\subge}_{g_{i_1}^{\sigma_1}} \right)
&\lesssim_C\,\, 2^{-(m-1)} s_3(C) + 2 \left( 1 - 2^{-(m-1)} \right) s_3(C)\\
&=\quad    2 \left( 1 - 2^{-m} \right) s_3(C) &\\
&\leq\quad 2 \left( 1 - 2^{-(l-1)} \right) s_3(C),
\end{align*}
since $m \leq l-1$. (Here the implicit multiplicative constant is the same as in Proposition \ref{regular_product} (i), and does not change after the induction step.)

\item Next we check (iv). By induction hypothesis (iv), we have
\[\begin{cases}
\displaystyle \left\| M(g') - \sum_{p=1}^{m} M(g_{i_p}^{\sigma_p}) \right\| \leq (m-1)\mu(2C) \vspace{2mm} \\
\displaystyle \left\| M(g'') - \sum_{p=m+1}^{l} M(g_{i_p}^{\sigma_p}) \right\| \leq (l-m-1)\mu(2C).
\end{cases}\]
If we take $s_3(C) \leq s_2(2C)$, then $g'$ and $g''$ satisfy Proposition \ref{invariant_additivity}, hence
\[\left\| M(g) - M(g') - M(g'') \right\| \leq \mu(2C).\]
Adding these three inequalities together, we get the desired conclusion.

\item It remains to check (v): let $h = g_{i'_1}^{\sigma'_1} \cdots g_{i'_{l'}}^{\sigma'_{l'}}$ be another cyclically reduced word (with $1 \leq l' \leq l$) such that $gh$ is also cyclically reduced. We need to check that the four pairs $(A^\subge_g, A^\suble_g)$, $(A^\subge_g, A^\suble_h)$, $(A^\subge_h, A^\suble_g)$ and $(A^\subge_h, A^\suble_h)$ are $2C$-non-degenerate. This follows by Lemma \ref{continuity_of_non_degeneracy} from the property (ii) (applied to both $g$ and $h$) and from the hypothesis (H2), provided we take $s_3(C)$ small enough. \qedhere
\end{itemize}
\end{proof}

\section{Construction of the group}
\label{sec:construction}

We now show (Lemma \ref{properly_discontinuous}) that if we take a group generated by a family of $\mathbb{R}$-regular, $C$-non-degenerate, sufficiently contracting maps with suitable Margulis invariants, it satisfies all of the conclusions of the Main Theorem, except Zariski-density. We then exhibit such a group that is also Zariski-dense (and thus prove the Main Theorem).

Recall (from Section \ref{sec:additivity}) that $w_0$ is some element of $G$ such that $w_0(\mathfrak{p}^+, \mathfrak{p}^-) = (\mathfrak{p}^-, \mathfrak{p}^+)$. Note that there is a nonzero vector $v \in \mathfrak{z}$ that is fixed by $-w_0$. Indeed, $w_0$ normalizes the space $\mathfrak{p}^+ \cap \mathfrak{p}^- = \mathfrak{l}$; hence it normalizes its split center $\mathfrak{a}$. Since $w_0$ exchanges positive and negative roots, the open Weyl chamber $\mathfrak{a}^+$ is stable by the involution $-w_0$. Since $\mathfrak{a}^+$ is a convex set, we may take $v = v' - w_0(v')$ for any $v'$ in $\mathfrak{a}^+$. Then we have indeed $v \neq 0$, $v \in \mathfrak{a} \subset \mathfrak{z}$ and $-w_0(v) = v$.

From now on, we fix a vector $M_0$ collinear to $v$ and such that $\|M_0\| = 2\mu(2C)$.
\begin{lemma}
\label{properly_discontinuous}
\begin{samepage}
Take any family $g_1, \ldots, g_k \in G \ltimes \mathfrak{g}$ satisfying the hypotheses (H1), (H2) and (H3) from Proposition \ref{Schottky_group}, and also the additional condition
\begin{enumerate}[(H4)]
\item For every $i$, $M(g_i) = M_0$.
\end{enumerate}
\end{samepage}
Then these maps generate a free group acting properly discontinuously on the affine space~$\mathfrak{g}_{\Aff}$.
\end{lemma}
\begin{proof}
To show that the group is free, simply remark that any non\-empty reduced word on the $g_i^{\pm 1}$ is conjugate to some cyclically reduced word, which, by Proposition \ref{Schottky_group}, is $\mathbb{R}$-regular and in particular different from the identity.

To show proper discontinuity, the first step is to prove the inequality~\eqref{eq:Margulis_unbounded} below, which says that cyclically reduced elements of the group have Margulis invariants that grow unboundedly. Take any cyclically reduced word $g = g_{i_1}^{\sigma_1} \cdots g_{i_l}^{\sigma_l}$. Then from Proposition~\ref{Schottky_group}, it follows that
\[\|M(g)\| \geq
\left\| \sum_{m=1}^{l} M(g_{i_m}^{\sigma_m}) \right\| - l\mu(2C).\]
On the other hand, we know that for every $i$ and $\sigma$, we have
\[M(g_i^\sigma) = M_0.\]
Indeed if $\sigma = +1$, this is true by hypothesis; if $\sigma = -1$, we have
\[M(g_i^{-1}) = -w_0(M(g_i)) = -w_0(M_0) = M_0,\]
by Proposition \ref{invariant_additivity} (i) and by definition of $M_0$. We conclude that
\begin{align}
\label{eq:Margulis_unbounded}
\|M(g)\| &\geq \|lM_0\| - l\mu(2C) \nonumber \\
  &= 2l\mu(2C) - l\mu(2C) \nonumber \\
  &= l\mu(2C).
\end{align}

Now let $K$ be any compact subset of the affine space $\mathfrak{g}_{\Aff}$, and suppose that $g$ is any reduced (not necessarily cyclically reduced) word on the $g_i^{\pm 1}$. We need to show that when $g$ is sufficiently long, we have $g(K) \cap K = \emptyset$.

Note first that it is always possible to find an index $i$ and a sign $\sigma$ such that $g_i^\sigma g$ is cyclically reduced. Then we have:
\begin{align*}
g(K) \cap K = \emptyset
  &\iff g_i^\sigma g(K) \cap g_i^\sigma(K) = \emptyset.
\end{align*}
Setting $K' = \bigcup_{i, \sigma} g_i^\sigma(K)$ (which is of course still compact), it is sufficient to prove that whenever $g$ is \emph{cyclically} reduced and sufficiently long, we have
\begin{equation}
\label{eq:gKcapK'}
g(K) \cap K' = \emptyset.
\end{equation}
Let $\phi_g$ be an optimal canonizing map for $g$, and let us define $\hat{\pi}_{\mathfrak{z}}$ on the whole space $\hat{\mathfrak{g}}$ as the (orthogonal) projection onto $\mathfrak{z} \oplus \mathbb{R}_0$ parallel to $\mathfrak{d} \oplus \mathfrak{n}^+ \oplus \mathfrak{n}^-$ (which may be seen as an affine map acting on $\mathfrak{g}_{\Aff}$). Then by definition of the Margulis invariant, we have
\begin{align*}
\hat{\pi}_{\mathfrak{z}} \circ \phi_g \left( g(K) \right)
  &= \tau_{M(g)} \circ \hat{\pi}_{\mathfrak{z}} \circ \phi_g(K) \\
  &= \hat{\pi}_{\mathfrak{z}} \circ \phi_g(K) + M(g).
\end{align*}
Now note that, on the one hand, $g$ is $2C$-non-degenerate by Proposition \ref{Schottky_group}, hence
\[\|\hat{\pi}_{\mathfrak{z}} \circ \phi_g(x - y)\| \leq \|\phi_g(x - y)\| \leq 2C\|x-y\|\]
for any $x, y \in \mathfrak{g}_{\Aff}$. On the other hand, recall the inequality \eqref{eq:Margulis_unbounded}:
\[\|M(g)\| \geq l\mu(2C),\]
where $l$ is the length of $g$. It follows that whenever
\[l > \frac{2C}{\mu(2C)}\max_{x \in K, y \in K'} \|x-y\|,\]
the images $\hat{\pi}_{\mathfrak{z}} \circ \phi_g(g(K))$ and $\hat{\pi}_{\mathfrak{z}} \circ \phi_g(K')$ are disjoint. This implies \eqref{eq:gKcapK'}, which in turn implies the conclusion.
\end{proof}

\begin{proof}[Proof of Main Theorem]
The strategy is now clear: we find a positive constant $C \geq 1$ and a family of maps $g_1, \ldots, g_k \in G \ltimes \mathfrak{g}$ (with $k \geq 2$) that satisfy the conditions (H1) through (H4) and whose linear parts generate a Zariski-dense subgroup of $G$, then we apply Lemma \ref{properly_discontinuous}. We proceed in several stages.
\begin{itemize}
\item By a result of Benoist (Lemma 7.2 in \cite{Ben96}), we may find a family of maps $\gamma_1, \ldots, \gamma_k \in G$ (that we shall see as elements of $G \ltimes \mathfrak{g}$, by identifying $G$ with the stabilizer of~$\mathbb{R}_0$), such that:
\begin{enumerate}[(i)]
\item Every $\gamma_i$ is $\mathbb{R}$-regular (this is (H1)).
\item For any two indices $i$, $i'$ and signs $\sigma$, $\sigma'$ such that $(i', \sigma') \neq (i, -\sigma)$, the spaces $V^{\subge}_{\gamma_i^\sigma}$ and $V^{\suble}_{\gamma_{i'}^{\sigma'}}$ are transverse.
\item Any single $\gamma_i$ generates a Zariski-connected group.
\item All of the $\gamma_i$ generate together a Zariski-dense subgroup of $G$.
\end{enumerate}
Note that Zariski-density is only possible if $k \geq 2$.
\item Clearly, every pair of transverse spaces is $C$-non-degenerate for some finite $C$; and here we have a finite number of such pairs. Hence if we choose some suitable value of $C$ (that we fix for the rest of this proof), the hypothesis (H2) becomes a direct consequence of the condition (ii) above.
\item From condition (iii) (Zariski-connectedness), it follows that any algebraic group containing some power $\gamma_i^N$ of some generator must actually contain the generator $\gamma_i$ itself. This allows us to replace every $\gamma_i$ by some power $\gamma_i^N$ without sacrificing condition (iv) (Zariski-density). Clearly, conditions (i), (ii) and (iii) are then preserved as well. If we choose $N$ large enough, we may suppose that the numbers $s(\gamma_i^{\pm 1})$ are as small as we wish: this gives us (H3). In fact, we shall suppose that for every $i$, we have $s(\gamma_i^{\pm 1}) \leq s_4(C)$ for an even smaller constant $s_4(C)$, to be specified soon.
\item To satisfy (H4), we replace the maps $\gamma_i$ by the maps
\[g_i := \tau_{\phi_i^{-1}(M_0)} \circ \gamma_i\]
(for $1 \leq i \leq k$), where $\phi_i$ is a canonizing map for $\gamma_i$.

We need to check that this does not break the first three conditions. Indeed, for every $i$, we have $\gamma_i = \ell(g_i)$; even better, since the translation vector $\phi_i^{-1}(M_0)$ lies in the subspace $V^{\sube}_{\gamma_i}$ stable by~$\gamma_i$, obviously the translation commutes with $\gamma_i$, hence $g_i$ has the same geometry as $\gamma_i$ (by this we mean that $A^{\subge}_{g_i} = A^{\subge}_{\gamma_i} = V^{\subge}_{\gamma_i} \oplus \mathbb{R}_0$ and $A^{\suble}_{g_i} = A^{\suble}_{\gamma_i} = V^{\suble}_{\gamma_i} \oplus \mathbb{R}_0$). Hence the $g_i$ still satisfy the hypotheses (H1) and (H2), but now we have $M(g_i) = M_0$ (this is (H4)). As for contraction strength, we have, by Lemma \ref{affine_to_vector}:
\[s(g_i) \lesssim_C s(\gamma_i)\|\tau_{M_0}\| \leq s_4(C)\|\tau_{M_0}\|,\]
and similarly for $g_i^{-1}$. Recall that $\|M_0\| = 2\mu(2C)$, hence $\|\tau_{M_0}\|$ depends only on~$C$. It follows that if we choose $s_4(C)$ small enough, the hypothesis (H3) is satisfied.

We conclude that the group generated by the elements $g_1, \ldots, g_k$ acts properly discontinuously (by Lemma \ref{properly_discontinuous}), is free (by the same result), nonabelian (since $k \geq 2$), and has linear part Zariski-dense in $G$, QED. \qedhere
\end{itemize}
\end{proof}

\section*{Acknowledgements}
I would like to thank my PhD advisor, Prof.~Yves Benoist, whose help while I was working on this paper has been invaluable to me.

\bibliographystyle{plain}
\bibliography{/home/ilia/Documents/Travaux_mathematiques/mybibliography.bib}
\end{document}